\documentclass[12pt]{amsart}                                                                                 %
\usepackage{color}                                                                                           %
\usepackage[colorlinks,linkcolor=blue,citecolor=blue]{hyperref}                                              %
\usepackage{graphicx}                                                                                        %
\usepackage[T1]{fontenc}                                                                                     %
\usepackage[latin1]{inputenc}                                                                                %
\usepackage{ae}                                                                                              %
\usepackage{amsmath}                                                                                         %
\usepackage{amsfonts}                                                                                        %
\usepackage{amssymb}                                                                                         %
\usepackage{amsxtra}                                                                                         %
\usepackage{dsfont}                                                                                          %
\usepackage{eucal}                                                                                           %
\usepackage{latexsym}                                                                                        %
\usepackage{amsthm}                                                                                          %
\usepackage{array}                                                                                           %
\usepackage{eqnarray}                                                                                        %
\usepackage{paralist}
\usepackage{cite}
\theoremstyle{plain}                                                                                         %
\newtheorem{thm}{Theorem}[section]                                                                         %
\newtheorem{prop}[thm]{Proposition}                                                                        %
\newtheorem{lem}[thm]{Lemma}                                                                             %
\newtheorem{cor}[thm]{Corollary}                                                                           %
\theoremstyle{definition}                                                                                    %
\newtheorem{dfn}[thm]{Definition}                                                                         %
\newtheorem{exam}[thm]{Example}                                                                         %
\theoremstyle{remark}                                                                                        %
\newtheorem{rmk}[thm]{Remark}
\pdfstringdef{\Title}{The Bartle--Dunford--Schwartz and the Dinculeanu--Singer theorems revisited}                                 %
\hypersetup{pdftitle=\Title}                                                                                 %
\pdfstringdef{\Author}{Fernando Mu\~{n}oz, Eve Oja, and C\'{a}ndido Pi\~{n}eiro}                                              %
\hypersetup{pdfauthor=\Author}                                                                               %
\hypersetup{pdfstartview=FitH}                                                                               %

\begin{document}

\baselineskip=17pt

\title[Classical representation theorems revisited]{The Bartle--Dunford--Schwartz and the Dinculeanu--Singer theorems revisited}

\author{Fernando Mu\~{n}oz, Eve Oja, and C\'{a}ndido Pi\~{n}eiro}



\address{Departamento de Matem\'aticas, Facultad de Ciencias Experimentales, Universidad de Huelva, Campus Universitario de El Carmen, 21071 Huelva, Spain}

\email{fmjimenez@dmat.uhu.es}

\address{Institute of Mathematics and Statistics, University of Tartu, J. Liivi 2, 50409 Tartu, Estonia; Estonian Academy of Sciences, Kohtu 6, 10130 Tallinn, Estonia}

\email{eve.oja@ut.ee}

\address{Departamento de Matem\'aticas, Facultad de Ciencias Experimentales, Universidad de Huelva, Campus Universitario de El Carmen, 21071 Huelva, Spain}

\email{candido@uhu.es}

\begin{abstract}
  Let $X$ and $Y$ be Banach spaces and let $\Omega$ be a compact Hausdorff space. Denote by $\mathcal{C}_{p}(\Omega,X)$ the space of $p$-continous $X$-valued functions, $1\leq p\leq \infty$. For operators $S\in\mathcal{L}(\mathcal{C}(\Omega),\mathcal{L}(X,Y))$ and $U\in\mathcal{L}(\mathcal{C}_{p}(\Omega,X),Y)$, we establish integral representation theorems with respect to a vector measure $m:\Sigma\rightarrow \mathcal{L}(X,Y^{**})$, where $\Sigma$ denotes the $\sigma$-algebra of Borel subsets of $\Omega$. The first theorem extends the classical Bartle--Dunford--Schwartz representation theorem. It is used to prove the second theorem, which extends the classical Dinculeanu--Singer representation theorem, also providing to it an alternative simpler proof. For the latter (and the main) result, we build the needed integration theory, relying on a new concept of the $q$-semivariation, $1\leq q\leq \infty$, of a vector measure $m:\Sigma\rightarrow \mathcal{L}(X,Y^{**})$.
\end{abstract}

\maketitle

\section{Introduction}\label{s1}

Let $X$ be a Banach space and let $\Omega$ be a compact Hausdorff space. The space of continuous functions from $\Omega$ into $X$ ($\mathds{K}$, respectively) is denoted by $\mathcal{C}(\Omega,X)$ ($\mathcal{C}(\Omega)$, respectively). We denote by $\Sigma$ the $\sigma$-algebra of Borel subsets of $\Omega$. The space of $\Sigma$-simple functions with values in $X$ and the Banach space of bounded $\Sigma$-measurable functions with values in $X$ (i.e., the space of functions from $\Omega$ into $X$ which are the uniform limit of a sequence of $\Sigma$-simple functions) are denoted by $\mathcal{S}(\Sigma,X)$ and $\mathcal{B}(\Sigma,X)$, respectively. In the case $X=\mathds{K}$, we abbreviate them to $\mathcal{S}(\Sigma)$ and $\mathcal{B}(\Sigma)$, respectively. It is well known that $\mathcal{C}(\Omega) \subset \mathcal{B}(\Sigma) \subset \mathcal{C}(\Omega)^{**}$ and, more generally, $\mathcal{C}(\Omega,X) \subset \mathcal{B}(\Sigma,X) \subset \mathcal{C}(\Omega,X)^{**}$ as closed subspaces.

Let $Y$ be a Banach space and denote by $\mathcal{L}(X,Y)$ the Banach space of bounded linear operators from $X$ into $Y$. Let $m:\Sigma\rightarrow Y$ be a vector measure of bounded semivariation. It is well known (see, e.g., \cite[pp. 6, 56, 153]{DU}) that the (elementary Bartle) integral $\int_{\Omega}(\cdot)\,dm$ is defined on $\mathcal{B}(\Sigma)$. (The definition passes from characteristic functions to functions in $\mathcal{S}(\Sigma)$ by linearity and to functions in $\mathcal{B}(\Sigma)$ by density.) By the Bartle--Dunford--Schwartz representation theorem, for every operator $S\in\mathcal{L}(\mathcal{C}(\Omega), Y)$ there exists a unique vector measure $m:\Sigma\rightarrow Y^{**}$ of bounded semivariation such that $S\varphi = \int_{\Omega}\varphi\,dm$ for all $\varphi\in\mathcal{C}(\Omega)$. The vector measure $m$ is called the \emph{representing measure} of $S$.

In \cite[Section 4]{MOP2}, this representation was extended from $Y\cong\mathcal{L}(\mathds{K},Y)$ to $\mathcal{L}(X,Y)$. Namely, in the case when $S\in\mathcal{L}(\mathcal{C}(\Omega),\mathcal{L}(X,Y))$, a vector measure $m:\Sigma\rightarrow \mathcal{L}(X,Y^{**})$ of bounded semivariation was built so that $S\varphi = \int_{\Omega}\varphi\,dm$ for all $\varphi\in\mathcal{C}(\Omega)$ (the construction of $m$ is recalled in Remark \ref{r:constructionofm} below). We define a \emph{representing measure} of $S\in\mathcal{L}(\mathcal{C}(\Omega),\mathcal{L}(X,Y))$ as a vector measure $m:\Sigma\rightarrow \mathcal{L}(X,Y^{**})$ of bounded semivariation which satisfies
\[
S\varphi = \int_{\Omega}\varphi\,dm\ \mbox{ for all } \varphi\in\mathcal{C}(\Omega).
\]

In Section \ref{s2}, we extend the Bartle--Dunford--Schwartz theorem, in all its aspects, to this general setting (see Theorem \ref{t:BDSt-S}). We also find a formula connecting the measure $m$ and the \emph{classical} representing measure $\mu:\Sigma\rightarrow \mathcal{L}(X,Y)^{**}$ of $S$ (as given by the Bartle--Dunford--Schwartz theorem) (see Corollary \ref{c:formula-m-mu}).

Results of Section \ref{s2} are applied in Section \ref{s4} to revisit the classical Dinculeanu--Singer representation theorem. By this theorem, for every operator $U\in\mathcal{L}(\mathcal{C}(\Omega, X), Y)$, there exists a unique vector measure $m:\Sigma\rightarrow \mathcal{L}(X,Y^{**})$ such that
\[
Uf = \int_{\Omega}f\,dm\ \mbox{ for all } f\in\mathcal{C}(\Omega,X),
\]
where the existence of the above integral requires from the measure $m$ that its $1$-semivariation (in our terminology; see Section \ref{s3} and, in particular, Example \ref{ex:q-semivar} showing that the $1$-semivariation coincides with the Gowurin--Dinculeanu semivariation) is bounded. 

In Section \ref{s4}, see Theorem \ref{t:Dinc-Sing-general}, which is the main result of this paper, we extend the Dinculeanu--Singer theorem, in all its aspects, from $\mathcal{C}(\Omega, X)$ to the Banach space $\mathcal{C}_{p}(\Omega, X)$ of $p$-continuous $X$-valued functions, where $1\leq p\leq \infty$ (studied in \cite{MOP1} and \cite{MOP2}; see Section \ref{s3} for the definition and needed properties), the spaces $\mathcal{C}_{p}(\Omega,X)$ being contained in $\mathcal{C}(\Omega, X)= \mathcal{C}_{\infty}(\Omega, X)$. However, this is not a routine extension: we do not follow the traditional proofs of the Dinculeanu--Singer theorem (see Remark \ref{r:proofDinculeanu--Singerth}), but we provide a handy alternative to them.

The scheme of our proof is very simple: for $U\in\mathcal{L}(\mathcal{C}_{p}(\Omega, X), Y)$, we consider the \emph{associated operator} $U^{\#}\in\mathcal{L}(\mathcal{C}(\Omega),\mathcal{L}(X,Y))$, defined by
\[
(U^{\#}\varphi)x = U(\varphi x), \ \varphi\in\mathcal{C}(\Omega), \ x\in X.
\]
By the above, we already have the representing measure $m:\Sigma\rightarrow \mathcal{L}(X,Y^{**})$ of $U^{\#}$. And we show (see Theorem \ref{t:mUassocismU}) that $Uf = \int_{\Omega} f\,dm$ for all $f\in\mathcal{C}_{p}(\Omega,X)$, meaning that our $m$ is also a representing measure of $U$.

On the other hand, one easily shows (see Proposition \ref{p:mUismUassoc}) that a representing measure $m:\Sigma\rightarrow \mathcal{L}(X,Y^{**})$ of $U\in\mathcal{L}(\mathcal{C}_{p}(\Omega, X), Y)$ is also a representing measure of $U^{\#}$. Therefore, since the representing measure of $U^{\#}$ is unique, also the representing measure of $U$ is unique. Hence, in the classical case when $U\in\mathcal{L}(\mathcal{C}(\Omega, X), Y)$, we regain the classical representing measure from the Dinculeanu--Singer theorem. Moreover, for the first time in the literature, a general formula, connecting the representing measure $m$ of $U$ and the classical representing measure $\mu:\Sigma\rightarrow\mathcal{L}(X,Y)^{**}$ of $U^{\#}$, is given (see Corollary \ref{c:relation-m-m-mu} and Remark \ref{r:relation-m-mu}).

In Section \ref{s3}, since the integration on $\mathcal{C}_{p}(\Omega,X)$ requires from the measure more than just the boundedness of its semivariation (but less than the integration on $\mathcal{C}(\Omega,X)$), we build the needed theory. For this end, we introduce the concept of the $q$\emph{-semivariation} of a vector measure $m:\Sigma\rightarrow \mathcal{L}(X,Y^{**})$ of bounded semivariation. This enables us to define an integral on $\mathcal{C}_{p}(\Omega,X)$ with values in $Y^{**}$, provided that the $p'$-semivariation of $m$ is bounded.

Finally, Section \ref{s5} is devoted to prove some qualitative complements to Theorem \ref{t:Dinc-Sing-general}, our extension of the Dinculeanu--Singer theorem, it uses results from the paper \cite{MOP2} by the authors and can be read just after Proposition \ref{p:mUismUassoc}.

Our notation is standard. Let $1\leq p \leq \infty$, and denote by $p'$ the conjugate index of $p$ (i.e., $1/p + 1/p' = 1$ with the convention $1/\infty = 0$). We consider Banach spaces over the same, either real or complex, field $\mathds{K}$. A Banach space $X$ will be regarded as a subspace of its bidual $X^{**}$ under the canonical isometric embedding $j_{X}:X\rightarrow X^{**}$. The closed unit ball of $X$ is denoted by $B_{X}$. The Banach space of all \emph{absolutely $p$-summable sequences} in $X$ is denoted by $\ell_{p}(X)$ and its norm by $\|\cdot\|_{p}$. The Banach operator ideal of absolutely $p$-summing operators is denoted by $\mathcal{P}_{p} = (\mathcal{P}_{p}, \|\cdot\|_{\mathcal{P}_{p}})$, and $\mathcal{L} = (\mathcal{L}, \|\cdot\|)$ is, as usual, the Banach operator ideal of bounded linear operators. We denote the characteristic function of $E\in\Sigma$ by $\chi_{E}$.

Our main reference to the vector measure theory is the book \cite{DU} by Diestel and Uhl. In particular, a \emph{vector measure} $m:\Sigma\rightarrow X$ is a finitely additive $X$-valued set function. The \emph{semivariation of} $m$ on $\Omega$ is denoted by $\|m\|(\Omega)$ and defined as
\[
\|m\|(\Omega)=\sup\Big\|\sum_{E_{i}\in\Pi} \varepsilon_{i}m(E_{i})\Big\|,
\]
where the supremum is taken over all finite partitions $\Pi=(E_{i})_{i=1}^{n}$ of $\Omega$ and all finite systems $(\varepsilon_{i})_{i=1}^{n}$ with $|\varepsilon_{i}|\leq1$, $1\leq i\leq n$, $n\in\mathds{N}$ (see, e.g., \cite[p. 4, Proposition 11]{DU}). If $\|m\|(\Omega)<\infty$, then $m$ is called a \emph{measure of bounded semivariation}. A vector measure $m:\Sigma\rightarrow X$ is \emph{bounded} if its range is bounded in $X$. This happens if and only if $m$ is of bounded semivariation (see, e.g., \cite[p. 4, Proposition 11]{DU}). Therefore, a vector measure of bounded semivariation is often called a \emph{bounded vector measure} (see, e.g., \cite[p. 5]{DU}), and we shall mainly use this term below.

\section{\texorpdfstring{Representing measure of $S\in\mathcal{L}(\mathcal{C}(\Omega),\mathcal{L}(X,Y))$}{Representing measure of S in L(C(K),L(X,Y))}}\label{s2}

Let $X$ and $Y$ be Banach spaces and let $\Omega$ be a compact Hausdorff space.

\begin{dfn}
Let $S\in\mathcal{L}(\mathcal{C}(\Omega),\mathcal{L}(X,Y))$. A \emph{representing measure} of $S$ is a bounded vector measure $m:\Sigma\rightarrow\mathcal{L}(X,Y^{**})$ which satisfies
\[
S\varphi = \int_{\Omega} \varphi\,dm  \ \ \ \mbox{ for all } \varphi\in\mathcal{C}(\Omega).
\]
\end{dfn}

As was mentioned in the Introduction, a representing measure $m:\Sigma\rightarrow \mathcal{L}(X,Y^{**})$ exists for every operator $S\in\mathcal{L}(\mathcal{C}(\Omega),\mathcal{L}(X,Y))$. Since we are going to use such a measure, it would be good (but not crucial) to know that it is unique. We start by a general observation that will also be used in Section \ref{s3}.

Let $m:\Sigma\rightarrow\mathcal{L}(X,Y^{**})$ be a bounded vector measure. Then, for every $x\in X$,
\[
m_{x} := m(\cdot)x :\Sigma\rightarrow Y^{**}
\]
is clearly a bounded vector measure. If $y^{*}\in Y^{*}$, then $x\otimes y^{*}\in\mathcal{L}(X,Y^{**})^{*}$ and for all $\varphi\in\mathcal{B}(\Sigma)$,
\begin{equation} \label{eq:intmu}
\langle \int_{\Omega}\varphi\,dm, x\otimes y^{*}\rangle = \int_{\Omega}\varphi\,d\mu_{x,y^{*}} = \langle y^{*}, \int_{\Omega}\varphi\,dm_{x}\rangle,
\end{equation}
where $\mu_{x,y^{*}}:=(x\otimes y^{*})m$, because
\[
\mu_{x,y^{*}}(E) = ((x\otimes y^{*})m)(E) = \langle m(E),x\otimes y^{*}\rangle  = \langle y^{*},m(E)x\rangle = \langle y^{*},m_{x}(E)\rangle,\ \ E\in\Sigma.
\]

Let $S\in\mathcal{L}(\mathcal{C}(\Omega),\mathcal{L}(X,Y))$. For every $x\in X$, define $S_{x}\in\mathcal{L}(\mathcal{C}(\Omega),Y)$ by
\[
S_{x}\varphi = (S\varphi)x,\ \ \varphi\in\mathcal{C}(\Omega).
\]

\begin{lem} \label{l:mxSx}
Let $X$ and $Y$ be Banach spaces and let $\Omega$ be a compact Hausdorff space. Let $m:\Sigma\rightarrow \mathcal{L}(X,Y^{**})$ be a representing measure of an operator $S\in\mathcal{L}(\mathcal{C}(\Omega), \mathcal{L}(X,Y))$. Then $m_{x}:\Sigma\rightarrow Y^{**}$ is the (classical) representing measure of the operator $S_{x}\in\mathcal{L}(\mathcal{C}(\Omega),Y)$ (given by the Bartle--Dunford--Schwartz theorem) and $\mu_{x,y^{*}} = S_{x}^{*}y^{*}$ for all $x\in X$ and $y^{*}\in Y^{*}$.
\end{lem}

\begin{proof}
For all $\varphi\in\mathcal{C}(\Omega)$, $x\in X$ and $y^{*}\in Y^{*}$, by \eqref{eq:intmu},
\[
\langle S_{x}\varphi,y^{*}\rangle = \langle (S\varphi)x,y^{*}\rangle = \langle S\varphi, x\otimes y^{*}\rangle = \langle y^{*}, \int_{\Omega}\varphi\,dm_{x}\rangle.
\]
This shows that $m_{x}$ is the representing measure of the operator $S_{x}\in\mathcal{L}(\mathcal{C}(\Omega),Y)$. Hence $S_{x}^{*}\in\mathcal{L}(Y^{*},\mathcal{C}(\Omega)^{*})$ and, by the Bartle--Dunford--Schwartz theorem, it is well known that $S_{x}^{*}y^{*} = \langle y^{*},m_{x}(\cdot)\rangle = \mu_{x,y^{*}}$.
\end{proof}

\begin{prop} \label{p:mUassocuniq}
Let $X$ and $Y$ be Banach spaces and let $\Omega$ be a compact Hausdorff space. Then the representing measure $m:\Sigma\rightarrow \mathcal{L}(X,Y^{**})$ of an operator $S\in\mathcal{L}(\mathcal{C}(\Omega), \mathcal{L}(X,Y))$ is unique.
\end{prop}

\begin{proof}
Let $m_{1},m_{2}:\Sigma\rightarrow \mathcal{L}(X,Y^{**})$ be two representing measures of an operator $S\in\mathcal{L}(\mathcal{C}(\Omega), \mathcal{L}(X,Y))$, and let $\mu_{x,y^{*}}^{i} = (x\otimes y^{*})m_{i}$, $i=1, 2$. We know from Lemma \ref{l:mxSx} that $\mu_{x,y^{*}}^{1} = S_{x}^{*}y^{*} = \mu_{x,y^{*}}^{2}$, giving that $\langle y^{*},m_{1}(E)x\rangle = \langle y^{*},m_{2}(E)x\rangle,$ for all $E\in\Sigma$, $x\in X$, and $y^{*}\in Y^{*}$. This means that $m_{1}=m_{2}$.
\end{proof}

Recalling that $\mathcal{L}(\mathds{K},Y)\cong Y$, the following result extends the classical Bartle--Dunford--Schwartz theorem (see, e.g., \cite[p. 152]{DU}) in all its aspects.

\begin{thm} \label{t:BDSt-S}
Let $X$ and $Y$ be Banach spaces and let $\Omega$ be a compact Hausdorff space.

$($\emph{a}$)$ Every operator $S\in\mathcal{L}(\mathcal{C}(\Omega),\mathcal{L}(X,Y))$ has a unique representing measure $m:\Sigma\rightarrow\mathcal{L}(X,Y^{**})$.

$($\emph{b}$)$ Assume that $m:\Sigma\rightarrow\mathcal{L}(X,Y^{**})$ is a bounded vector measure. Then, there exists an operator $S\in\mathcal{L}(\mathcal{C}(\Omega),\mathcal{L}(X,Y))$ such that  $m$ is its representing measure if and only if for all $x\in X$,
\[
\mu_{x,y^{*}} = \langle y^{*},m_{x}(\cdot)\rangle \in\mathcal{C}(\Omega)^{*},\ \ y^{*}\in Y^{*},
\]
and the map $Y^{*}\rightarrow \mathcal{C}(\Omega)^{*}$, $y^{*}\mapsto \langle y^{*},m_{x}(\cdot)\rangle$, is linear, bounded, and weak*-to-weak* continuous.

In this case, $m_{x}:\Sigma\rightarrow Y^{**}$ is the representing measure of the operator $S_{x}\in\mathcal{L}(\mathcal{C}(\Omega),Y)$ and $\mu_{x,y^{*}} = S_{x}^{*}y^{*}$ for all $x\in X$ and $y^{*}\in Y^{*}$, the equality $\|S\| = \|m\|(\Omega)$ holds, and the measure $m:\Sigma\rightarrow \mathcal{L}(X,Y^{**}) = (X\hat{\otimes}_{\pi}Y^{*})^{*}$ is weak*-countably additive.
\end{thm}

\begin{proof}
(a) The existence of a representing measure was proved in \cite[Section 4]{MOP2}. The (simple) construction of this measure is recalled in Remark \ref{r:constructionofm}. Its uniqueness comes from Proposition \ref{p:mUassocuniq}.

(b) By Lemma \ref{l:mxSx}, we know that $m_{x}:\Sigma\rightarrow Y^{**}$ is the representing measure of $S_{x}\in\mathcal{L}(\mathcal{C}(\Omega),Y)$. Then $S^{*}_{x}\in\mathcal{L}(Y^{*},\mathcal{C}(\Omega)^{*})$ is weak*-to-weak* continuous. Also, by Lemma \ref{l:mxSx}, $S^{*}_{x}y^{*}=\mu_{x,y^{*}} = \langle y^{*},m_{x}(\cdot)\rangle$. This proves the ``only if'' part.

For the ``if'' part, let $S$ be the restriction to $\mathcal{C}(\Omega)$ of the integration operator $\int_{\Omega}\varphi\,dm$, $\varphi\in\mathcal{B}(\Sigma)$. Then $S\in\mathcal{L}(\mathcal{C}(\Omega),\mathcal{L}(X,Y^{**}))$. It remains to show that $(S\varphi)x\in Y$ for all $\varphi\in\mathcal{C}(\Omega)$ and $x\in X$. By \eqref{eq:intmu},
\[
\langle y^{*}, (S\varphi)x\rangle = \langle S\varphi, x\otimes y^{*}\rangle = \langle y^{*}, \int_{\Omega}\varphi\,dm_{x} \rangle,\ \ y^{*}\in Y^{*}.
\]
Hence, $(S\varphi)x\in Y^{**}$ is the composition of the weak*-to-weak* continuous map $y^{*}\mapsto \langle y^{*},m_{x}(\cdot)\rangle$ from $Y^{*}$ to $\mathcal{C}(\Omega)^{*}$ and the weak* continuous functional $\mu\mapsto \int_{\Omega}\varphi\,d\mu$ on $\mathcal{C}(\Omega)^{*}$. Therefore, $(S\varphi)x$ is a weak* continuous functional on $Y^{*}$, and $(S\varphi)x\in Y$ as desired.

For the ``in this case'' part, the first two claims come from Lemma \ref{l:mxSx}. The third claim was proved in \cite[Proposition 4.1]{MOP2}; for an alternative proof, see Corollary \ref{c:normSandm} below.

Finally, to show the weak*-countable additivity of $m$, let $(E_{n})$ be a sequence of pairwise disjoint members of $\Sigma$. Denote $f_{k} := \sum_{n=1}^{k}m(E_{n})$, $k\in\mathds{N}$, and $f:=m(\bigcup_{n=1}^{\infty}E_{n})$. By the countable additivity of $\mu_{x,y^{*}}$, we have
\[
\langle x\otimes y^{*}, f\rangle = \mu_{x,y^{*}}(\bigcup_{n=1}^{\infty}E_{n}) = \sum_{n=1}^{\infty} \mu_{x,y^{*}}(E_{n}) = \lim_{k \rightarrow \infty} \langle x\otimes y^{*}, f_{k}\rangle
\]
for all $x\in X$ and $y^{*}\in Y^{*}$. Since also the sequence $(f_{k})$ is bounded (in fact, $\|f_{k}\|\leq\|m\|(\Omega)$), $f_{k}\rightarrow f$ pointwise on $X\hat{\otimes}_{\pi}Y^{*}$. This means that
\[
\langle u, m(\bigcup_{n=1}^{\infty}E_{n}) \rangle = \sum_{n=1}^{\infty}\langle u, m(E_{n}) \rangle
\]
for all $u\in X\hat{\otimes}_{\pi}Y^{*}$, as desired.
\end{proof}

In the classical case when $S\in\mathcal{L}(\mathcal{C}(\Omega), Y)$ and $m:\Sigma\rightarrow Y^{**}$ is its representing measure, the integration operator $\hat{S}\in\mathcal{L}(\mathcal{B}(\Sigma), Y^{**})$, $\hat{S}\varphi = \int_{\Omega}\varphi\,dm$, $\varphi\in\mathcal{B}(\Sigma)$, extends the operator $S$ from $\mathcal{C}(\Omega)$ to $\mathcal{B}(\Sigma)$, where $\mathcal{C}(\Omega)$ sits as a closed subspace. And, in turn, $S^{**}\in\mathcal{L}(\mathcal{C}(\Omega)^{**}, Y^{**})$ extends the operator $\hat{S}$ from $\mathcal{B}(\Sigma)$ to $\mathcal{C}(\Omega)^{**}$, where $\mathcal{B}(\Sigma)$ sits as a closed subspace.

In \cite[the proof of Proposition 4.1]{MOP2}, we pointed out that a similar phenomenon occurs also in the general case when $S\in\mathcal{L}(\mathcal{C}(\Omega),\mathcal{L}(X,Y))$. In the following, we shall make this precise.

Let $S\in\mathcal{L}(\mathcal{C}(\Omega),\mathcal{L}(X,Y))$ and let $m:\Sigma\rightarrow\mathcal{L}(X,Y^{**})$ be its representing measure. Then, as above, the integration operator $\hat{S}\in\mathcal{L}(\mathcal{B}(\Sigma), \mathcal{L}(X,Y^{**}))$ extends the operator $S$. More precisely, let
\[
J:\mathcal{L}(X, Y)\rightarrow \mathcal{L}(X, Y^{**}),\ J(A) = j_{Y}A,\ A\in\mathcal{L}(X, Y),
\]
be the natural isometric embedding. Then
\[
JS=\hat{S}|_{\mathcal{C}(\Omega)}.
\]
To understand in which sense $S^{**}$ ``extends'' $\hat{S}$, recall that $\mathcal{L}(X,Y^{**}) = (X\hat{\otimes}_{\pi}Y^{*})^{*}$ as Banach spaces ($\pi$ denotes the projective tensor norm, as usual), and put
\[
P := (j_{X\hat{\otimes}_{\pi}Y^{*}})^{*}.
\]
Then $P$ is the (natural) projection from $\mathcal{L}(X,Y^{**})^{**} = (X\hat{\otimes}_{\pi}Y^{*})^{***}$ onto $\mathcal{L}(X,Y^{**}) = (X\hat{\otimes}_{\pi}Y^{*})^{*}$.

\begin{thm} \label{t:hatS}
Let $X$ and $Y$ be Banach spaces and let $\Omega$ be a compact Hausdorff space. Assume that $S\in\mathcal{L}(\mathcal{C}(\Omega),\mathcal{L}(X,Y))$ and let $m:\Sigma\rightarrow\mathcal{L}(X,Y^{**})$ be its representing measure. Then, with the above notation,
\begin{equation} \label{eq:pjs}
\hat{S} = PJ^{**}S^{**}|_{\mathcal{B}(\Sigma)}.
\end{equation}
\end{thm}

\begin{proof}
It suffices to verify that
\[
\hat{S}\chi_{E} = PJ^{**}S^{**}\chi_{E}\ \ \mbox{ for all }E\in\Sigma.
\]
Then by linearity, \eqref{eq:pjs} holds on $\mathcal{S}(\Sigma)$, and by density, \eqref{eq:pjs} holds on $\mathcal{B}(\Sigma)$.

For this end, in turn, it suffices to verify that
\begin{equation} \label{eq:pjs-2}
\langle x\otimes y^{*}, \hat{S}\chi_{E}\rangle = \langle x\otimes y^{*},PJ^{**}S^{**}\chi_{E}\rangle, \ \ x\in X, y^{*}\in Y^{*}.
\end{equation}

For the left-hand side of \eqref{eq:pjs-2}, we have
\begin{equation} \label{eq:hatS1}
\langle x\otimes y^{*}, \hat{S}\chi_{E}\rangle = \langle x\otimes y^{*},m(E)\rangle = \langle y^{*}, m(E)x\rangle = \mu_{x,y^{*}}(E).
\end{equation}

For the right-hand side of \eqref{eq:pjs-2}, we have, considering $\mathcal{C}(\Omega)^{*}$ embedded in $\mathcal{B}(\Sigma)^{*}$,
\[
\langle x\otimes y^{*},PJ^{**}S^{**}\chi_{E}\rangle = \langle j_{X\hat{\otimes}_{\pi}Y^{*}} (x\otimes y^{*}),J^{**}S^{**}\chi_{E}\rangle
\]
\[
= \langle \chi_{E},S^{*}J^{*}j_{X\hat{\otimes}_{\pi}Y^{*}} (x\otimes y^{*})\rangle = \langle \chi_{E},S^{*}(x\otimes y^{*})\rangle,
\]
because
\[
\langle A,J^{*}j_{X\hat{\otimes}_{\pi}Y^{*}}(x\otimes y^{*})\rangle = \langle x\otimes y^{*},J(A)\rangle = \langle y^{*}, j_{Y}Ax\rangle = \langle Ax, y^{*}\rangle = \langle A,x\otimes y^{*}\rangle
\]
for all $A\in\mathcal{L}(X,Y)$. But it is clear that
\[
S^{*}(x\otimes y^{*}) = S_{x}^{*}y^{*}.
\]
Indeed, for every $\varphi \in \mathcal{C}(\Omega)$, we have
\[
\langle \varphi,S^{*}(x\otimes y^{*})\rangle = \langle S\varphi, x\otimes y^{*}\rangle = \langle (S\varphi)x,y^{*}\rangle = \langle S_{x}\varphi, y^{*}\rangle = \langle \varphi,S_{x}^{*}y^{*}\rangle.
\]
Therefore, using that $S_{x}^{*}y^{*}= \mu_{x,y^{*}}$ (see Lemma \ref{l:mxSx}), we obtain
\begin{equation} \label{eq:hatS2}
\langle x\otimes y^{*},PJ^{**}S^{**}\chi_{E} \rangle = \langle \chi_{E},\mu_{x,y^{*}}\rangle = \mu_{x,y^{*}}(E).
\end{equation}

From \eqref{eq:hatS1} and \eqref{eq:hatS2}, we get that \eqref{eq:pjs-2} holds.
\end{proof}

\begin{rmk}
The above proof does not require the uniqueness of the representing measure $m:\Sigma\rightarrow\mathcal{L}(X,Y^{**})$ of $S\in\mathcal{L}(\mathcal{C}(\Omega),\mathcal{L}(X,Y))$ nor how the measure $m$ is built.
\end{rmk}

\begin{cor} [{see \cite[Proposition 4.1]{MOP2}}] \label{c:normSandm}
Let $X$ and $Y$ be Banach spaces and let $\Omega$ be a compact Hausdorff space. Assume that $S\in\mathcal{L}(\mathcal{C}(\Omega),\mathcal{L}(X, Y))$ and let $m:\Sigma\rightarrow \mathcal{L}(X, Y^{**})$ be its representing measure. Then $\|m\|(\Omega)=\|S\|$.
\end{cor}

\begin{proof}
It is well known that $\|m\|(\Omega) = \|\hat{S}\|$ (see, e.g., \cite[p. 6, Theorem 13]{DU}). But
\[
\|S\| = \|JS\| = \|\hat{S}|_{\mathcal{C}(\Omega)}\| \leq \|\hat{S}\| = \|PJ^{**}S^{**}|_{\mathcal{B}(\Sigma)}\| \leq \|S\|.
\]
\end{proof}

Thanks to Theorem \ref{t:hatS}, we have an alternative proof for the uniqueness of the representing measure $m$.

\begin{cor} [{see Proposition \ref{p:mUassocuniq}}]
Let $X$ and $Y$ be Banach spaces and let $\Omega$ be a compact Hausdorff space. Then the representing measure $m:\Sigma\rightarrow \mathcal{L}(X,Y^{**})$ of an operator $S\in\mathcal{L}(\mathcal{C}(\Omega), \mathcal{L}(X,Y))$ is unique.
\end{cor}

\begin{proof}
Let $m:\Sigma\rightarrow \mathcal{L}(X,Y^{**})$ be a representing measure of $S\in\mathcal{L}(\mathcal{C}(\Omega), \mathcal{L}(X,Y))$. Then for all $E\in\Sigma$, we  have
\[
m(E) = \hat{S}\chi_{E} = PJ^{**}S^{**}\chi_{E}.
\]
Hence, if $m_{1},m_{2}:\Sigma\rightarrow \mathcal{L}(X,Y^{**})$ are representing measures of $S$, then $m_{1}(E) = m_{2}(E)$ for all $E\in\Sigma$.
\end{proof}

For $S\in\mathcal{L}(\mathcal{C}(\Omega), \mathcal{L}(X,Y))$, together with its representing measure $m:\Sigma\rightarrow \mathcal{L}(X,Y^{**})$, there also exists its classical representing measure, say $\mu:\Sigma\rightarrow \mathcal{L}(X,Y)^{**}$ (given by the Bartle--Dunford--Schwartz theorem). Let $\hat{\hat{S}}\in\mathcal{L}(\mathcal{B}(\Sigma),\mathcal{L}(X,Y)^{**})$ denote the corresponding integration operator, i.e., $\hat{\hat{S}} = \int_{\Omega}\varphi\,d\mu$, $\varphi\in\mathcal{B}(\Sigma)$. As is well known (this was also mentioned above), $S^{**}|_{\mathcal{B}(\Sigma)} = \hat{\hat{S}}$. Hence, Theorem \ref{t:hatS} tells us that
\[
\hat{S} = PJ^{**}\hat{\hat{S}}.
\]
On characteristic functions, this gives the following formula \eqref{eq:relation-m-mu} which connects the measures $m$ and $\mu$.

\begin{cor} \label{c:formula-m-mu}
Let $X$ and $Y$ be Banach spaces and let $\Omega$ be a compact Hausdorff space. Assume that $S\in\mathcal{L}(\mathcal{C}(\Omega), \mathcal{L}(X,Y))$, and let $m:\Sigma\rightarrow \mathcal{L}(X,Y^{**})$ and $\mu:\Sigma\rightarrow \mathcal{L}(X,Y)^{**}$ be its representing measures. Then
\begin{equation} \label{eq:relation-m-mu}
m(E) = PJ^{**}\mu(E)\ \ \mbox{ for all $E\in\Sigma$. }
\end{equation}

Moreover, if $S$ is weakly compact, then $m$ takes its values in $\mathcal{L}(X,Y)$, and the measures $m$ and $\mu$ coincide. In this case, the measure $m:\Sigma\rightarrow\mathcal{L}(X,Y)$ is countably additive and regular.
\end{cor}

\begin{proof}
By the above, only the ``moreover'' part needs a proof. From the Bartle--Dunford--Schwartz theory \cite{BDS} (see, e.g., \cite[p. 153, Theorem 5]{DU}), it is well known that if $S\in\mathcal{L}(\mathcal{C}(\Omega), \mathcal{L}(X,Y))$ is weakly compact, then $\mu$ takes its values in $\mathcal{L}(X,Y)$ and $\mu:\Sigma\rightarrow\mathcal{L}(X,Y)$ is countably additive. It is also regular (see \cite[p. 159, Corollary 14]{DU}). But for every $A\in\mathcal{L}(X,Y)$, considering $\mathcal{L}(X,Y)$ embedded in $\mathcal{L}(X,Y)^{**}$, we have
\[
\langle x\otimes y^{*}, PJ^{**}(A)\rangle = \langle j_{X\hat{\otimes}_{\pi}Y^{*}}(x\otimes y^{*}),J^{**}(A)\rangle
\]
\[
= \langle A, J^{*}j_{X\hat{\otimes}_{\pi}Y^{*}}(x\otimes y^{*})\rangle = \langle x\otimes y^{*},j_{Y}A \rangle
\]
for all $x\in X$ and $y^{*}\in Y^{*}$, implying that $PJ^{**}(A) = j_{Y}A$. Therefore, by \eqref{eq:relation-m-mu},
\[
m(E) = PJ^{**}\mu(E) = j_{Y}\mu(E)
\]
for all $E\in\Sigma$. This means that $m$ takes its values in $\mathcal{L}(X,Y)$ and considering $\mathcal{L}(X,Y)$ embedded in $\mathcal{L}(X,Y^{**})$, the measures $m:\Sigma\rightarrow\mathcal{L}(X,Y)$ and $\mu:\Sigma\rightarrow\mathcal{L}(X,Y)$ coincide.
\end{proof}

The next example shows that the fact that the representing measure $m:\Sigma\rightarrow\mathcal{L}(X,Y^{**})$ of an operator $S\in\mathcal{L}(\mathcal{C}(\Omega),\mathcal{L}(X,Y))$ takes its values in $\mathcal{L}(X,Y)$ does not imply the weak compactness of the operator $S$.

\begin{exam}
Denote by $\beta\mathds{N}$ the \v{C}ech--Stone compactification of $\mathds{N}$. As is well known, $\mathcal{C}(\beta\mathds{N}) = \ell_{\infty}$. Consider the identity operator $I\in\mathcal{L}(\ell_{\infty}, \ell_{\infty}) = \mathcal{L}(\mathcal{C}(\beta\mathds{N}), \mathcal{L}(\ell_{1},\mathds{K}))$. Since $\ell_{\infty}$ is not reflexive, $I$ is a non-weakly compact operator. However, its representing measure $m:\Sigma\rightarrow \mathcal{L}(\ell_{1},\mathds{K}^{**})$ takes its values in $\mathcal{L}(\ell_{1},\mathds{K})=\mathcal{L}(\ell_{1},\mathds{K}^{**})$.
\end{exam}

\begin{rmk}
Let $S\in\mathcal{L}(\mathcal{C}(\Omega),\mathcal{L}(X,Y))$ and let $m:\Sigma\rightarrow\mathcal{L}(X,Y^{**})$ be the representing measure of $S$. By definition of the measures $m_{x}$, the measure $m$ takes its values in $\mathcal{L}(X,Y)$ if and only if all $m_{x}:\Sigma\rightarrow Y^{**}$, $x\in X$, take their values in $Y$. Since $m_{x}$ is the representing measure of $S_{x}$, by the Bartle--Dunford--Schwartz theory (see, e.g., \cite[p. 153, Theorem 5]{DU}), this is equivalent to the fact that all operators $S_{x}\in\mathcal{L}(\mathcal{C}(\Omega),Y)$, $x\in X$, are weakly compact. This clearly happens when $S$ is weakly compact.
\end{rmk}

\begin{rmk} \label{r:constructionofm}
Let $S\in\mathcal{L}(\mathcal{C}(\Omega),\mathcal{L}(X,Y))$. In the above, we only needed (and used) the fact (from \cite{MOP2}) that a representing measure $m:\Sigma\rightarrow \mathcal{L}(X,Y^{**})$ exists for $S$. For completeness, let us recall how $m$ is built in \cite[Section 4]{MOP2}. Let $S_{x}\in\mathcal{L}(\mathcal{C}(\Omega), Y)$, $x\in X$, be defined (as above) by $S_{x}\varphi=(S\varphi)x$, $\varphi\in\mathcal{C}(\Omega)$, and let $m_{x}:\Sigma\rightarrow Y^{**}$ be its representing measure (given by the Bartle--Dunford--Schwartz theorem). Then $m:\Sigma\rightarrow \mathcal{L}(X,Y^{**})$ is defined by
\[
\langle y^{*},m(E)x \rangle = \langle y^{*},m_{x}(E)\rangle,\ \ E\in\Sigma,
\]
for all $x\in X$ and $y^{*}\in Y^{*}$.
\end{rmk}

\section{\texorpdfstring{Integration of $p$-continuous vector-valued functions with respect to an operator-valued measure}{Integration of p-continuous vector-valued functions with respect to an operator-valued measure}}\label{s3}

Let $X$ and $Y$ be Banach spaces and let $\Omega$ be a compact Hausdorff space. Let $1\leq p\leq\infty$. The Banach space $\mathcal{C}_{p}(\Omega,X)$ of \emph{$p$-continuous} $X$-valued functions \cite{MOP1} is formed by all $f\in\mathcal{C}(\Omega,X)$ such that $f(\Omega)$ is $p$-compact (i.e., there exists a sequence $(x_{n})\in\ell_{p}(X)$ (or $(x_{n})\in c_{0}(X)$ when $p=\infty$) such that $f(\Omega)\subset \{ \sum_{n}\alpha_{n}x_{n} \, : \, (\alpha_{n})\in B_{\ell_{p'}} \}$). It follows from properties of $p$-compactness that $\mathcal{C}_{p}(\Omega, X) \subset \mathcal{C}_{q}(\Omega, X)$ if $p\leq q$, and $\mathcal{C}_{\infty}(\Omega, X) = \mathcal{C}(\Omega, X)$. The space $\mathcal{C}_{p}(\Omega, X)$ becomes a Banach space endowed with the norm
\[
\|f\|_{\mathcal{C}_{p}(\Omega, X)}=\inf \|(x_{n})\|_{p},
\]
where the infimum is taken over all sequences $(x_{n})\in\ell_{p}(X)$ (or $(x_{n})\in c_{0}(X)$ when $p=\infty$) such that $f(\Omega)\subset \{ \sum_{n}\alpha_{n}x_{n} \, : \, (\alpha_{n})\in B_{\ell_{p'}} \}$, and $\mathcal{C}_{\infty}(\Omega, X) = \mathcal{C}(\Omega, X)$ as Banach spaces (see \cite[Proposition 3.6]{MOP1}).

By Grothendieck's classics \cite{G2} (see, e.g., \cite[pp. 49--50]{R}), we know that
\[
\mathcal{C}(\Omega,X) = \mathcal{C}(\Omega)\hat{\otimes}_{\varepsilon}X
\]
as Banach spaces, where $\varepsilon$ denotes the injective tensor norm, under the canonical isometric isomorphism $\varphi x \leftrightarrow \varphi\otimes x$, $\varphi\in\mathcal{C}(\Omega)$ and $x\in X$. One of the main results of \cite{MOP1} is that
\[
\mathcal{C}_{p}(\Omega,X)=\mathcal{C}(\Omega)\hat{\otimes}_{d_{p}}X
\]
as Banach spaces, where $d_{p}$ denotes the right Chevet--Saphar tensor norm (see \cite{S} or, e.g., \cite[Chapter 6]{R} for the definition and properties; we do not need the definition in this paper).

Let $m:\Sigma\rightarrow \mathcal{L}(X,Y)$ be a vector measure. It is well known that the ``algebraic'' integral $\int_{\Omega}(\cdot)\,dm$ is defined on $\mathcal{S}(\Sigma, X)$. (The definition passes from vector-valued characteristic functions $\chi_{E}x$, $E\in\Sigma$, $x\in X$, to functions in $\mathcal{S}(\Sigma, X)$ by linearity.)

The classical Dinculeanu--Singer representation theorem requires the integration on $\mathcal{C}(\Omega, X)$. The corresponding integral was built by Dinculeanu (see \cite[II.7.1, II.9,1, and p. 398, Theorem 9]{D}; an early idea of this integral can be found in \cite{Go} and \cite{B}). In fact, the Dinculeanu integral was built on $\mathcal{B}(\Sigma, X)$, where $\mathcal{C}(\Omega, X)$ sits as a closed subspace, and then restricted to $\mathcal{C}(\Omega, X)$. On the other hand, the Dinculeanu integral restricted to $\mathcal{S}(\Sigma, X)$ coincides with the ``algebraic'' integral.

The existence of the Dinculeanu integral requires from $m$ much more than does the existence of the elementary Bartle integral, where the semivariation $\|m\|(\Omega)$ was needed to be finite. Namely, a much bigger ``semivariation'' than $\|m\|(\Omega)$ must be finite. Let us call it the \emph{Gowurin--Dinculeanu semivariation} (it was introduced by Gowurin \cite{Go} and deeply studied by Dinculeanu (see, e.g., \cite[I.4]{D})).

To be able to integrate on $\mathcal{C}_{p}(\Omega, X)$, we shall need an ``intermediate semivariation'', depending on $p$, which, in the ``limit'' cases for $\mathcal{C}_{1}(\Omega, X)$ and $\mathcal{C}_{\infty}(\Omega, X) = \mathcal{C}(\Omega, X)$, coincides with the (usual) semivariation $\|m\|(\Omega)$ and the Gowurin--Dinculeanu semivariation, respectively (see Example \ref{ex:q-semivar} below).

Before introducing our ``intermediate semivariation'', we shall need the description of the dual space $\mathcal{C}_{p}(\Omega, X)^{*}$ as a space of operators from $\mathcal{C}(\Omega)$ to $X^{*}$. Recall (see, e.g., \cite[p. 142]{R}) that the dual space operator ideal (we follow the terminology of \cite{OR}) of the Chevet--Saphar tensor norm $d_{p}$ coincides with $\mathcal{P}_{p'}$, i.e., $(Z\hat{\otimes}_{d_{p}}X)^{*} = \mathcal{P}_{p'}(Z,X^{*})$ as Banach spaces (here $Z$ is an arbitrary Banach space). (Recall that $\mathcal{P}_{q} = (\mathcal{P}_{q}, \|\cdot\|_{\mathcal{P}_{q}})$, $1\leq q\leq \infty$, denotes the Banach operator ideal of absolutely $q$-summing operators.) Since $\mathcal{C}_{p}(\Omega, X) = \mathcal{C}(\Omega)\hat{\otimes}_{d_{p}}X$ as Banach spaces, we have
\[
\mathcal{C}_{p}(\Omega, X)^{*} = \mathcal{P}_{p'}(\mathcal{C}(\Omega),X^{*}),
\]
as Banach spaces, under the duality
\[
\langle \varphi x, T\rangle = \langle x, T\varphi \rangle,\ \ \varphi\in\mathcal{C}(\Omega), x\in X, T\in\mathcal{P}_{p'}(\mathcal{C}(\Omega),X^{*}).
\]

Let $m:\Sigma\rightarrow\mathcal{L}(X,Y^{**})$ be a bounded vector measure. Notice that this clearly encompasses the seemingly more general case when $m$ takes its values in $\mathcal{L}(X,Y)$, because $Y$ is canonically embedded in $Y^{**}$. Then, for every $y^{*}\in Y^{*}$,
\[
m_{y^{*}} := (m(\cdot))^{*}y^{*} :\Sigma\rightarrow X^{*}
\]
is clearly a bounded vector measure. From the beginning of Section \ref{s2}, we know that
\[
\langle x, m_{y^{*}}(E)\rangle = \langle x, (m(E))^{*}y^{*}\rangle = \langle y^{*},m(E)x\rangle = \mu_{x,y^{*}}(E),
\]
and therefore, for all $\varphi\in\mathcal{B}(\Sigma)$,
\begin{equation} \label{eq:intmystar}
\langle \int_{\Omega}\varphi\,dm,x\otimes y^{*}\rangle = \int_{\Omega}\varphi\,d\mu_{x,y^{*}} = \langle x, \int_{\Omega}\varphi\,dm_{y^{*}}\rangle.
\end{equation}

Denote by $I_{y^{*}}$ the restriction of the latter integral from $\mathcal{B}(\Sigma)$ to $\mathcal{C}(\Omega)$, i.e., for every $y^{*}\in Y^{*}$,
\[
I_{y^{*}}\varphi = \int_{\Omega}\varphi\,dm_{y^{*}},\ \ \varphi\in\mathcal{C}(\Omega).
\]
Then $I_{y^{*}}\in\mathcal{L}(\mathcal{C}(\Omega),X^{*})$ and $m_{y^{*}}:\Sigma\rightarrow X^{*}$ is its representing measure.

Let $1\leq q\leq\infty$. We define the \emph{$q$-semivariation} $\|m\|_{q}(\Omega)$ of a bounded vector measure $m:\Sigma\rightarrow\mathcal{L}(X,Y^{**})$ by
\[
\|m\|_{q}(\Omega) = \sup_{y^{*}\in B_{Y^{*}}} \|I_{y^{*}}\|_{\mathcal{P}_{q}}.
\]
We say that a bounded vector measure $m:\Sigma\rightarrow\mathcal{L}(X,Y^{**})$ is of \emph{bounded $q$-semivariation} if $\|m\|_{q}(\Omega)<\infty$. It follows from the inclusion theorem for absolutely $q$-summing operators (see, e.g., \cite[p. 39, Theorem 2.8]{DJT}) that
\[
\|m\|_{\infty}(\Omega) \leq \|m\|_{q}(\Omega) \leq \|m\|_{p}(\Omega) \leq \|m\|_{1}(\Omega)\ \ \mbox{ if } 1\leq  p\leq q\leq \infty.
\]

\begin{exam} \label{ex:q-semivar}
Let $X$ and $Y$ be Banach spaces and let $\Omega$ be a compact Hausdorff space. Let $m:\Sigma\rightarrow \mathcal{L}(X,Y^{**})$ be a bounded vector measure. Then $\|m\|_{\infty}(\Omega) = \|m\|(\Omega)$, the semivariation of $m$, and $\|m\|_{1}(\Omega)$ coincides with the Gowurin--Dinculeanu semivariation.

\begin{proof}
Let $y^{*}\in Y^{*}$. Since $(\mathcal{P}_{\infty}, \|\cdot\|_{\mathcal{P}_{\infty}}) =(\mathcal{L}, \|\cdot\|)$, we have that $\|I_{y^{*}}\|_{\mathcal{P}_{\infty}} = \|I_{y^{*}}\|$. And since $m_{y^{*}}$ is the representing measure of $I_{y^{*}}\in\mathcal{L}(\mathcal{C}(\Omega),X^{*})$, we have that $\|I_{y^{*}}\| = \|m_{y^{*}}\|(\Omega)$, by the Bartle--Dunford--Schwartz theorem. Therefore
\[
\|m\|_{\infty}(\Omega) = \sup_{y^{*}\in B_{Y^{*}}}\|m_{y^{*}}\|(\Omega)
\]
\[
= \sup\Big\{\Big\|\sum_{E_{i}\in\Pi}\varepsilon_{i}m_{y^{*}}(E_{i})\Big\|\,:\,y^{*}\in B_{Y^{*}}, |\varepsilon_{i}|\leq 1, \Pi\Big\}
\]
\[
= \sup\Big\{\Big| \langle x, \sum_{E_{i}\in\Pi}\varepsilon_{i}m_{y^{*}}(E_{i})\rangle \Big|\,:\,x\in B_{X}, y^{*}\in B_{Y^{*}}, |\varepsilon_{i}|\leq 1, \Pi\Big\}
\]
\[
= \sup\Big\{\Big| \langle y^{*}, \Big(\sum_{E_{i}\in\Pi}\varepsilon_{i}m(E_{i})\Big)x\rangle \Big|\,:\,x\in B_{X}, y^{*}\in B_{Y^{*}}, |\varepsilon_{i}|\leq 1,\Pi\Big\}
\]
\[
= \sup\Big\{\Big\| \Big(\sum_{E_{i}\in\Pi}\varepsilon_{i}m(E_{i})\Big)x\Big\|\,:\,x\in B_{X}, |\varepsilon_{i}|\leq 1, \Pi\Big\}
\]
\[
= \sup\Big\{\Big\| \sum_{E_{i}\in\Pi}\varepsilon_{i}m(E_{i})\Big\|\,:\, |\varepsilon_{i}|\leq 1, \Pi\Big\} = \|m\|(\Omega).
\]

We know that $\mathcal{P}_{1}(\mathcal{C}(\Omega),X^{*}) = \mathcal{C}_{\infty}(\Omega, X)^{*} = \mathcal{C}(\Omega, X)^{*}$. We also know that $I_{y^{*}}\in\mathcal{L}(\mathcal{C}(\Omega),X^{*})$ is absolutely summing, i.e., $I_{y^{*}}\in\mathcal{P}_{1}(\mathcal{C}(\Omega),X^{*})$ if and only if its representing measure $m_{y^{*}}$ is of bounded variation, and in this case, $\|I_{y^{*}}\|_{\mathcal{P}_{1}} = |m_{y^{*}}|(\Omega)$ (see, e.g., \cite[p. 162, Theorem 3]{DU}). Hence
\begin{equation} \label{eq:defsemivDU-DIN}
\|m\|_{1}(\Omega) = \sup_{y^{*}\in B_{Y^{*}}} |m_{y^{*}}|(\Omega),
\end{equation}
which, thanks to \cite[p. 55, Proposition 5]{D}, coincides with the Gowurin--Dinculeanu semivariation of $m$. Let us recall that in \cite[p. 181]{DU}, formula \eqref{eq:defsemivDU-DIN} is taken as the definition of the Gowurin--Dinculeanu semivariation of $m$.
\end{proof}
\end{exam}

Below, we shall need the following result which, among others, may be used for calculating $\|m\|_{q}(\Omega)$. For $y^{*}\in Y^{*}$, let
\[
\hat{I}_{y^{*}}:=\int_{\Omega}(\cdot)\,dm_{y^{*}} \in \mathcal{L}(\mathcal{B}(\Sigma),X^{*})
\]
denote the integration operator with respect to $m_{y^{*}}$.

\begin{prop} \label{p:qsummingnormofI}
Let $X$ and $Y$ be Banach spaces and let $\Omega$ be a compact Hausdorff space. Let $1\leq q\leq\infty$. Assume that $m:\Sigma\rightarrow \mathcal{L}(X,Y^{**})$ is a bounded vector measure. Then
\[
\|\hat{I}_{y^{*}}\|_{\mathcal{P}_{q}} = \|I_{y^{*}}\|_{\mathcal{P}_{q}} \ \ \ \mbox{ for all } y^{*}\in Y^{*}.
\]
\end{prop}

\begin{proof}
Since $\mathcal{C}(\Omega)\subset\mathcal{B}(\Sigma)\subset\mathcal{C}(\Omega)^{**}$ as closed subspaces, $\hat{I}_{y^{*}}$ is an extension of $I_{y^{*}}$, and $(I_{y^{*}})^{**}$ is an extension of $\hat{I}_{y^{*}}$, we have that
\[
\|I_{y^{*}}\|_{\mathcal{P}_{q}} \leq \|\hat{I}_{y^{*}}\|_{\mathcal{P}_{q}} \leq \|(I_{y^{*}})^{**}\|_{\mathcal{P}_{q}}.
\]
Hence, if $\|I_{y^{*}}\|_{\mathcal{P}_{q}} = \infty$, then also $\|\hat{I}_{y^{*}}\|_{\mathcal{P}_{q}} = \infty$. If $\|I_{y^{*}}\|_{\mathcal{P}_{q}} < \infty$, i.e., $I_{y^{*}}$ is absolutely $q$-summing, then also $(I_{y^{*}})^{**}$ is, and in this case, $\|(I_{y^{*}})^{**}\|_{\mathcal{P}_{q}} = \|I_{y^{*}}\|_{\mathcal{P}_{q}}$ (see, e.g., \cite[p. 50, Proposition 2.19]{DJT}). Therefore $\|I_{y^{*}}\|_{\mathcal{P}_{q}} = \|\hat{I}_{y^{*}}\|_{\mathcal{P}_{q}}$, as desired.
\end{proof}

It is well known (see, e.g., \cite[p. 11]{R}) that $\mathcal{B}(\Sigma)\otimes X\subset \mathcal{B}(\Sigma, X)$ as a linear subspace, under the algebraic identification $\varphi\otimes x \leftrightarrow \varphi x$. This is used in the following result.

\begin{thm} \label{t:def-int-dp}
Let $X$ and $Y$ be Banach spaces and let $\Omega$ be a compact Hausdorff space. Let $1\leq p\leq\infty$. Assume that $m:\Sigma\rightarrow \mathcal{L}(X,Y^{**})$ is a bounded vector measure. Then the formula
\begin{equation} \label{eq:int-dp}
\int_{\Omega}(\varphi x)\,dm = \Big(\int_{\Omega}\varphi\,dm\Big)x,\ \ \varphi\in\mathcal{B}(\Sigma), x\in X,
\end{equation}
defines an integral on $\mathcal{B}(\Sigma)\hat{\otimes}_{d_{p}} X$ with respect to $m$ if and only if $\|m\|_{p'}(\Omega)<\infty$. In this case, the integration operator $\hat{U}$ belongs to $\mathcal{L}(\mathcal{B}(\Sigma)\hat{\otimes}_{d_{p}} X,Y^{**})$, $\|\hat{U}\|=\|m\|_{p'}(\Omega)$, the restriction of $\hat{U}$ to $\mathcal{S}(\Sigma, X) = \mathcal{S}(\Sigma)\otimes X$ concides with the ``algebraic'' integral, and $\hat{U}^{*}y^{*} = \hat{I}_{y^{*}}$ for all $y^{*}\in Y^{*}$. 

Moreover, the measure $m$ takes its values in $\mathcal{L}(X,Y)$ if and only if the integration operator $\hat{U}$ takes its values in $Y$.
\end{thm}

\begin{proof}
First of all, notice that if the main part of the theorem holds true, then the ``if'' part of the ``moreover'' part is clear from \eqref{eq:int-dp}. Indeed, assume that $\textrm{ran}\,\hat{U}\subset Y$. Since
\[
m(E) = \int_{\Omega}\chi_{E}\,dm \ \ \ \mbox{ for all } E\in\Sigma,
\]
by \eqref{eq:int-dp}, we have that
\[
m(E)x = \int_{\Omega}(\chi_{E}x)\,dm = \hat{U}(\chi_{E}\otimes x)\in Y \ \ \ \mbox{ for all } E\in\Sigma \mbox{ and } x\in X.
\]
This means that $\textrm{ran}\,m\subset \mathcal{L}(X,Y)$.

To prove the theorem and to encompass also the ``only if'' part of the ``moreover'' part, let $W:=Y^{**}$ or $W:=Y$.

On the right-hand side of \eqref{eq:int-dp}, the integral is just the elementary Bartle integral with respect to $m$. Denote by $\hat{S}\in\mathcal{L}(\mathcal{B}(\Sigma),\mathcal{L}(X,W))$ this integration operator. Since, as is well known, $\mathcal{L}(Z,\mathcal{L}(X,W))$ is canonically isometrically isomorphic to $\mathcal{L}(Z\otimes_{\pi}X,W) = \mathcal{L}(Z\hat{\otimes}_{\pi}X,W)$ (for any Banach spaces $X$, $W$, and $Z$), there exists a unique linear operator $\hat{U}:\mathcal{B}(\Sigma)\otimes X\rightarrow W$ such that
\[
\hat{U}(\varphi\otimes x) = (\hat{S}\varphi)x,\ \ \varphi\in\mathcal{B}(\Sigma), x\in X,
\]
and $\hat{U}\in\mathcal{L}(\mathcal{B}(\Sigma)\otimes_{\pi}X, W)$. Hence, by \eqref{eq:int-dp},
\[
\int_{\Omega}(\varphi x)\,dm = \hat{U}(\varphi\otimes x),\ \ \varphi\in\mathcal{B}(\Sigma), x\in X,
\]
giving that
\[
\int_{\Omega}f\,dm = \hat{U}f,\ \ f\in\mathcal{S}(\Sigma,X) = \mathcal{S}(\Sigma)\otimes X.
\]

It remains to prove that
\begin{equation} \label{eq:nu}
\nu := \sup\{\|\hat{U}v\|\,:\, v\in \mathcal{B}(\Sigma)\otimes X, \|v\|_{d_{p}}\leq 1 \} = \|m\|_{p'}(\Omega).
\end{equation}
Then, in the case when $\|m\|_{p'}(\Omega)<\infty$ or, equivalently, $\hat{U}\in\mathcal{L}( \mathcal{B}(\Sigma)\otimes_{d_{p}}X,W)$, by passing to the unique continuous linear extension of $\hat{U}$, we get that $\hat{U}\in\mathcal{L}( \mathcal{B}(\Sigma)\hat{\otimes}_{d_{p}}X,W)$ and $\|\hat{U}\| = \|m\|_{p'}(\Omega)$. Therefore, the integral $\int_{\Omega}(\cdot)\,dm$ is defined on $\mathcal{B}(\Sigma)\hat{\otimes}_{d_{p}}X$ by
\[
\int_{\Omega}v\,dm = \hat{U}v,\ \ v\in\mathcal{B}(\Sigma)\hat{\otimes}_{d_{p}}X.
\]

Let us now prove equality \eqref{eq:nu}. Fix an arbitrary $y^{*}\in Y^{*}$. Then, for all $\varphi\in\mathcal{B}(\Sigma)$ and $x\in X$, by \eqref{eq:intmystar}, we have
\begin{equation} \label{eq:hatI=hatU}
\begin{split}
\langle x,\hat{I}_{y^{*}}\varphi\rangle = & \langle \hat{S}\varphi, x\otimes y^{*}\rangle = \langle y^{*},(\hat{S}\varphi)x\rangle
\\
= \langle y^{*},\hat{U}(\varphi\otimes x)\rangle &= \langle \varphi\otimes x,\hat{U}^{*}y^{*}\rangle = \langle x, (\hat{U}^{*}y^{*})\varphi\rangle;
\end{split}
\end{equation}
for the two last equalities, recall that we have
\[
\hat{U}^{*}\in\mathcal{L}(W^{*}, (\mathcal{B}(\Sigma)\otimes_{\pi}X)^{*}),
\]
so that $\hat{U}^{*}y^{*}\in(\mathcal{B}(\Sigma)\otimes_{\pi}X)^{*}=\mathcal{L}(\mathcal{B}(\Sigma),X^{*})$. Therefore, $\hat{I}_{y^{*}}=\hat{U}^{*}y^{*}$ and thus
\[
\|\hat{I}_{y^{*}}\|_{\mathcal{P}_{p'}} = \|\hat{U}^{*}y^{*}\|_{\mathcal{P}_{p'}} = \sup\{ |\langle v, \hat{U}^{*}y^{*}\rangle |\,:\,v\in \mathcal{B}(\Sigma)\otimes X, \|v\|_{d_{p}}\leq 1\}
\]
\[
= \sup\{ |\langle y^{*},\hat{U}v\rangle |\,:\,v\in \mathcal{B}(\Sigma)\otimes X, \|v\|_{d_{p}}\leq 1\} \leq \nu \|y^{*}\|.
\]
Hence,
\[
\nu \geq \|m\|_{p'}(\Omega).
\]

For the reverse inequality, let $v=\sum_{i=1}^{n}\varphi_{i}\otimes x_{i}\in\mathcal{B}(\Sigma)\otimes_{d_{p}}X$. For any $y^{*}\in Y^{*}$, by \eqref{eq:hatI=hatU}, we have
\[
|\langle y^{*},\hat{U}v\rangle| = \Big|\langle y^{*}, \sum_{i=1}^{n}\hat{U}(\varphi_{i}\otimes x_{i})\rangle| = \Big|\sum_{i=1}^{n}\langle y^{*},\hat{U}(\varphi_{i}\otimes x_{i})\rangle|
\]
\[
= \Big|\sum_{i=1}^{n}\langle x_{i},\hat{I}_{y^{*}}\varphi_{i}\rangle| \leq \|(x_{i})_{i=1}^{n}\|_{p} \|(\hat{I}_{y^{*}}\varphi_{i})_{i=1}^{n}\|_{p'} \leq \|(x_{i})_{i=1}^{n}\|_{p} \|\hat{I}_{y^{*}}\|_{\mathcal{P}_{p'}}\|(\varphi_{i})_{i=1}^{n}\|_{p'}^{w}.
\]
Taking first the infimum over all the representations of $v\in \mathcal{B}(\Sigma)\otimes_{d_{p}}X$ and then the supremum over $y^{*}\in B_{Y^{*}}$, by Proposition \ref{p:qsummingnormofI}, we obtain that
\[
\|\hat{U}v\| \leq \|m\|_{p'}(\Omega) \|v\|_{d_{p}},
\]
hence,
\[
\nu \leq \|m\|_{p'}(\Omega),
\]
and \eqref{eq:nu} holds.

Finally, if $\hat{U}\in\mathcal{L}(\mathcal{B}(\Sigma)\hat{\otimes}_{d_{p}}X,W)$, then we have
\[
\hat{U}^{*}\in\mathcal{L}(W^{*},(\mathcal{B}(\Sigma)\hat{\otimes}_{d_{p}}X)^{*}) = \mathcal{L}(W^{*},\mathcal{P}_{p'}(\mathcal{B}(\Sigma),X^{*})),
\]
and equalities \eqref{eq:hatI=hatU} hold true, giving that $\hat{I}_{y^{*}} = \hat{U}^{*}y^{*}$ for all $y^{*}\in Y^{*}$.
\end{proof}

As we mentioned in the beginning of this section, $\mathcal{C}_{p}(\Omega, X) = \mathcal{C}(\Omega)\hat{\otimes}_{d_{p}}X$ as Banach spaces, under the identification $\varphi x \leftrightarrow \varphi\otimes x$. On the other hand, let us observe that $\mathcal{C}(\Omega)\hat{\otimes}_{d_{p}}X$ \emph{is a closed subspace of} $\mathcal{B}(\Sigma)\hat{\otimes}_{d_{p}}X$. Indeed, it is well known that $\mathcal{C}(\Omega)^{*}$ is isometrically isomorphic to an $L_{1}(\mu)$-space for some measure $\mu$, i.e., $\mathcal{C}(\Omega)$ is an \emph{$L_{1}$-predual space}. Thanks to Fakhoury \cite[Corollary 3.3]{F} and Grothendieck \cite[Theorem 1]{G3} (see,  e.g., \cite[pp. 76, 81]{DFS}),  $L_{1}$-predual spaces are ideals in their ``superspaces'' (for more details, see \cite[p. 49]{LLO}). In particular, $\mathcal{C}(\Omega)$ is an ideal in $\mathcal{B}(\Sigma)$. But then (see \cite[Proposition 2.4]{OR}) $\mathcal{C}(\Omega)\otimes_{d_{p}}X$ is a subspace of $\mathcal{B}(\Sigma)\hat{\otimes}_{d_{p}}X$, and therefore $\mathcal{C}(\Omega)\hat{\otimes}_{d_{p}}X=\overline{\mathcal{C}(\Omega)\otimes_{d_{p}}X}$ is a closed subspace of $\mathcal{B}(\Sigma)\hat{\otimes}_{d_{p}}X$.

Therefore $\mathcal{C}_{p}(\Omega, X)$ \emph{is a closed subspace of} $\mathcal{B}(\Sigma)\hat{\otimes}_{d_{p}}X$, and Theorem \ref{t:def-int-dp} almost immediately yields the integration result below (Theorem \ref{t:def-int-U}).

Let $m:\Sigma\rightarrow \mathcal{L}(X,Y^{**})$ be a vector measure of bounded $p'$-semivariation. Denote by
\[
U:=\hat{U}|_{\mathcal{C}_{p}(\Omega, X)} = \hat{U}|_{\mathcal{C}(\Omega)\hat{\otimes}_{d_{p}}X}
\]
the restriction to $\mathcal{C}_{p}(\Omega, X)$ of the integration operator $\hat{U}$ given by Theorem \ref{t:def-int-dp}, i.e.,
\[
Uf = \int_{\Omega}f\,dm,\ \ \ f\in\mathcal{C}_{p}(\Omega, X).
\]

\begin{thm} \label{t:def-int-U}
Let $X$ and $Y$ be Banach spaces and let $\Omega$ be a compact Hausdorff space. Let $1\leq p\leq\infty$. Assume that $m:\Sigma\rightarrow \mathcal{L}(X,Y^{**})$ is a vector measure of bounded $p'$-semivariation. Then the formula \eqref{eq:int-dp} defines an integral on $\mathcal{C}_{p}(\Omega, X)$ with respect to $m$, the integration operator $U$ belongs to $\mathcal{L}(\mathcal{C}_{p}(\Omega, X), Y^{**})$, $\|U\|=\|m\|_{p'}(\Omega)$, and $U^{*}y^{*} = I_{y^{*}}$ for all $y^{*}\in Y^{*}$. 

Moreover, if the integration operator $U$ takes its values in $Y$, in particular, this is the case when the measure $m$ takes its values in $\mathcal{L}(X,Y)$, then
\[
U^{**}(\chi_{E}\otimes x) = m(E)x \ \ \ \mbox{ for all } E\in\Sigma \mbox{ and } x\in X,
\]
where $\chi_{E}\otimes x \in \mathcal{C}_{p}(\Omega,X)^{**}$ is defined in the canonical way:
\[
\langle A, \chi_{E}\otimes x\rangle = \langle A^{*}x, \chi_{E}\rangle, A\in\mathcal{P}_{p'}(\mathcal{C}(\Omega),X^{*}) = \mathcal{C}_{p}(\Omega,X)^{*}.
\]
\end{thm}

\begin{proof}
For the main part of the theorem, in view of Theorem \ref{t:def-int-dp}, we only need to show that $\|U\| \geq \|m\|_{p'}(\Omega)$ (because $\|U\|\leq\|\hat{U}\|=\|m\|_{p'}(\Omega)$) and $U^{*}y^{*}=I_{y^{*}}$ for all $y^{*}\in Y^{*}$.

Let $y^{*}\in Y^{*}$. Using that $U\in\mathcal{L}(\mathcal{C}(\Omega)\hat{\otimes}_{d_{p}}X,Y^{**})$ and $(\mathcal{C}(\Omega)\hat{\otimes}_{d_{p}}X)^{*} = \mathcal{P}_{p'}(\mathcal{C}(\Omega),X^{*})$, so that $U^{*}\in\mathcal{L}(Y^{***},\mathcal{P}_{p'}(\mathcal{C}(\Omega),X^{*}))$, we get from \eqref{eq:hatI=hatU} that
\[
\langle x, I_{y^{*}}\varphi \rangle = \langle y^{*},U(\varphi\otimes x)\rangle = \langle \varphi\otimes x, U^{*}y^{*}\rangle = \langle x,(U^{*}y^{*})\varphi\rangle
\]
for all $x\in X$ and $\varphi\in\mathcal{C}(\Omega)$. Therefore $I_{y^{*}} = U^{*}y^{*}$ and
\[
\|I_{y^{*}}\|_{\mathcal{P}_{p'}} = \|U^{*}y^{*}\|_{\mathcal{P}_{p'}} \leq \|U^{*}\|\|y^{*}\| = \|U\|\|y^{*}\|
\]
for all $y^{*}\in Y^{*}$. This yields that
\[
\|m\|_{p'}(\Omega) = \sup_{y^{*}\in B_{Y^{*}}} \|I_{y^{*}}\|_{\mathcal{P}_{p'}} \leq \|U\|.
\]

Now, for the ``moreover'' part, assume that $\textrm{ran}\,U\subset Y$. Then $U\in\mathcal{L}(\mathcal{C}_{p}(\Omega,X),Y)$. Let $E\in\Sigma$, $x\in X$, and $y^{*}\in Y^{*}$. Then,
\[
\langle y^{*}, U^{**}(\chi_{E}\otimes x)\rangle = \langle U^{*}y^{*}, \chi_{E}\otimes x\rangle = \langle I_{y^{*}},\chi_{E}\otimes x\rangle = \langle (I_{y^{*}})^{*}x,\chi_{E}\rangle
\]
\[
= \langle x,(I_{y^{*}})^{**}\chi_{E}\rangle = \langle x, \hat{I}_{y^{*}}\chi_{E}\rangle = \langle y^{*}, \hat{U}(\chi_{E}\otimes x)\rangle,
\]
where the last equality holds by \eqref{eq:hatI=hatU}. Therefore, $U^{**}(\chi_{E}\otimes x) = \hat{U}(\chi_{E}\otimes x)$ for all $E\in\Sigma$ and $x\in X$. But, by \eqref{eq:int-dp}, we have that
\[
\hat{U}(\chi_{E}\otimes x) = \Big(\int_{\Omega}\chi_{E}\,dm\Big)x = m(E)x,
\]
proving that $U^{**}(\chi_{E}\otimes x) = m(E)x$ for all $E\in\Sigma$ and $x\in X$.

Finally, let us recall from Theorem \ref{t:def-int-dp} that $\textrm{ran}\,m\subset\mathcal{L}(X,Y)$ if and only if $\textrm{ran}\,\hat{U}\subset Y$. Hence, in this case, $\textrm{ran}\,U\subset Y$.
\end{proof}

\begin{rmk}
Form Example \ref{ex:q-semivar} and Theorem \ref{t:def-int-U}, it is clear that, in the special case when $p=\infty$, our integral coincides with the Dinculeanu integral from \cite{D}.
\end{rmk}

\begin{rmk}
Our notion of the $q$-semivariation is different from the notion ``$q$-semivariation'' introduced in Dinculeanu's book \cite[p. 246]{D}. Let us call the latter ``the Dinculeanu $q$-semivariation''. Its definition is as follows.

Let $1\leq q\leq \infty$ and let $\mu:\Sigma\rightarrow \mathds{R}$ be a positive finite measure; we may assume that $\mu(\Omega)=1$. For a vector measure $m:\Sigma\rightarrow \mathcal{L}(X,Y)$, the \emph{Dinculeanu $q$-semivariation} on $\Omega$ (see \cite[p. 246]{D}) is defined by
\[
\tilde{m}_{q}(\Omega) = \sup \Big\{\Big\| \sum_{E_{i}\in\Pi} m(E_{i})x_{i}\Big\|\Big\},
\]
where the supremum is taken over all finite partitions $\Pi = (E_{i})_{i=1}^{n}$ of $\Omega$ and all finite systems $(x_{i})_{i=1}^{n}\subset X$ such that $\|\sum_{i=1}^{n}\chi_{E_{i}}x_{i}\|_{L_{q'}(\mu,X)} \leq 1$, $n\in\mathds{N}$. This notion is used in \cite[II.13]{D} to obtain the integral representation of an operator $U\in\mathcal{L}(L_{p}(\mu, X), Y)$, $1\leq p < \infty$, with respect to a vector measure $m:\Sigma\rightarrow \mathcal{L}(X,Y)$ such that $\tilde{m}_{p'}(\Omega)<\infty$.

It can be easily verified that $\|m\|_{1}(\Omega)\leq \tilde{m}_{1}(\Omega)$ and $\|m\|_{1}(\Omega)= \tilde{m}_{1}(\Omega)$ if $m$ is absolutely continuous with respect to $\mu$ (see \cite[p. 246]{D}). Since also $\tilde{m}_{1}(\Omega) \leq \tilde{m}_{q}(\Omega)$ (see \cite[p. 247]{D}), we have that
\[
\|m\|_{q}(\Omega) \leq \|m\|_{1}(\Omega) \leq \tilde{m}_{1}(\Omega) \leq \tilde{m}_{q}(\Omega).
\]
\end{rmk}

\section{\texorpdfstring{Representing measure of $U\in\mathcal{L}(\mathcal{C}_{p}(\Omega,X),Y)$}{Representing measure of U in L(Cp(K,X),Y)}}\label{s4}

Let $X$ and $Y$ be Banach spaces and let $\Omega$ be a compact Hausdorff space. Let $1\leq p\leq\infty$. Basing on Theorem \ref{t:def-int-U}, we may give the following definition whose special case when $p=\infty$, thanks to Example \ref{ex:q-semivar}, coincides with the classical one, known from the Dinculeanu--Singer theorem.

\begin{dfn} \label{d:representingmU}
Let $U\in\mathcal{L}(\mathcal{C}_{p}(\Omega,X),Y)$. A \emph{representing measure} of $U$ is a vector measure $m:\Sigma\rightarrow \mathcal{L}(X,Y^{**})$ of bounded $p'$-semivariation which satisfies
\begin{equation} \label{eq:intU}
Uf = \int_{\Omega}f\,dm \ \mbox{ for all } f\in\mathcal{C}_{p}(\Omega,X).
\end{equation}
\end{dfn}

\begin{rmk}
In the classical case of $\mathcal{C}(\Omega,X) = \mathcal{C}_{\infty}(\Omega,X)$, Definition \ref{d:representingmU} differs from the definition of representing measure by Brooks and Lewis \cite[Definition 2.9]{BrLe}. Namely, we do not require that the measures $m_{y^{*}}:\Sigma\rightarrow X^{*}$, $y^{*}\in Y^{*}$ (see Section \ref{s3}) were regular. They have this regularity property thanks to Theorem \ref{t:Dinc-Sing-general} below. More precisely, the regularity holds whenever $p\neq 1$ (and this condition is essential by Example \ref{ex:regularityofm}).
\end{rmk}

\begin{thm} \label{t:mUassocismU}
Let $X$ and $Y$ be Banach spaces and let $\Omega$ be a compact Hausdorff space. Let $1\leq p\leq\infty$. Assume that $U\in\mathcal{L}(\mathcal{C}_{p}(\Omega,X), Y)$, and let $m:\Sigma\rightarrow \mathcal{L}(X,Y^{**})$ be the representing measure of the associated operator $U^{\#}\in\mathcal{L}(\mathcal{C}(\Omega), \mathcal{L}(X, Y))$. Then $m$ is a representing measure of $U$, $I_{y^{*}}=U^{*}y^{*}\in\mathcal{P}_{p'}(\mathcal{C}(\Omega),X^{*})$ for all $y^{*}\in Y^{*}$, $\|U\|=\|m\|_{p'}(\Omega)$, and
\[
U^{**}(\chi_{E}\otimes x) = m(E)x \ \ \ \mbox{ for all } E\in\Sigma \mbox{ and } x\in X.
\]
\end{thm}

\begin{proof}
We know that $m:\Sigma\rightarrow \mathcal{L}(X,Y^{**})$ is a bounded vector measure. For all $\varphi\in\mathcal{C}(\Omega)$, $x\in X$, and $y^{*}\in Y^{*}$, by \eqref{eq:intmystar}, we have that
\[
\langle x, I_{y^{*}}\varphi \rangle = \langle \int_{\Omega}\varphi\,dm, x\otimes y^{*}\rangle = \langle U^{\#}\varphi, x\otimes y^{*} \rangle = \langle y^{*}, (U^{\#}\varphi)x \rangle
\]
\[
= \langle y^{*}, U(\varphi\otimes x) \rangle = \langle U^{*}y^{*},\varphi\otimes x \rangle = \langle x,(U^{*}y^{*})\varphi \rangle.
\]
Hence $I_{y^{*}} = U^{*}y^{*}$ for all $y^{*}\in Y^{*}$. Therefore $I_{y^{*}}\in\mathcal{P}_{p'}(\mathcal{C}(\Omega),X^{*})$ and
\[
\|m\|_{p'}(\Omega) = \sup_{y^{*}\in B_{Y^{*}}} \|U^{*}y^{*}\|_{\mathcal{P}_{p'}} = \|U^{*}\| = \|U\|<\infty.
\]
Since $m$ is of bounded $p'$-semivariation, the formula \eqref{eq:int-dp} defines
\[
\int_{\Omega}(\cdot)\,dm\in\mathcal{L}(\mathcal{C}_{p}(\Omega, X), Y^{**})
\]
(see Theorem \ref{t:def-int-U}). We only need to show \eqref{eq:intU}, because then also the last claim holds true thanks to Theorem \ref{t:def-int-U}.

Let $\varphi\in\mathcal{C}(\Omega)$ and $x\in X$. Then
\[
U(\varphi x) = (U^{\#}\varphi)x = \Big(\int_{\Omega}\varphi\,dm\Big)x = \int_{\Omega}(\varphi x)\,dm
\]
by \eqref{eq:int-dp}. It is well known (see, e.g., \cite[p. 11]{R}) that $\mathcal{C}(\Omega)\otimes X\subset \mathcal{C}_{p}(\Omega, X)$ as a linear subspace (under the algebraic identification $\varphi\otimes x \leftrightarrow \varphi x$ that was used in Section \ref{s3}). Therefore, by linearity, \eqref{eq:intU} holds for every $f\in\mathcal{C}(\Omega)\otimes X$. If now $f\in\mathcal{C}_{p}(\Omega,X) = \mathcal{C}(\Omega)\hat{\otimes}_{d_{p}} X$ is arbitrary, then $f = \lim_{n}f_{n}$ in $\mathcal{C}_{p}(\Omega, X)$ for some $f_{n}\in\mathcal{C}(\Omega)\otimes X$. Hence,
\[
Uf = \lim_{n}Uf_{n} = \lim_{n}\int_{\Omega}f_{n}\,dm \mbox{ in } Y.
\]
On the other hand, by the definition of the integral,
\[
\int_{\Omega}f\,dm = \lim_{n}\int_{\Omega}f_{n}\,dm \mbox{ in } Y^{**}.
\]
Consequently, \eqref{eq:intU} holds.
\end{proof}

Theorem \ref{t:mUassocismU} shows that a representing measure of $U\in\mathcal{L}(\mathcal{C}_{p}(\Omega,X),Y)$ may be defined as the representing measure of its associated operator $U^{\#}$. Now we see that this is, in fact, the unique way to define a representing measure $m:\Sigma\rightarrow \mathcal{L}(X,Y^{**})$ for $U$.

\begin{prop} \label{p:mUismUassoc}
Let $X$ and $Y$ be Banach spaces and let $\Omega$ be a compact Hausdorff space. Let $1\leq p\leq\infty$. Assume that $m:\Sigma\rightarrow \mathcal{L}(X,Y^{**})$ is a representing measure of $U\in\mathcal{L}(\mathcal{C}_{p}(\Omega,X), Y)$. Then $m$ is the representing measure of $U^{\#}$.
\end{prop}

\begin{proof}
Let $\varphi\in\mathcal{C}(\Omega)$ and $x\in X$. Then, using \eqref{eq:int-dp}, we have
\[
(U^{\#}\varphi)x = U(\varphi x) = \int_{\Omega}(\varphi x)\,dm = \Big(\int_{\Omega}\varphi\,dm\Big)x.
\]
Thus
\[
U^{\#}\varphi = \int_{\Omega}\varphi\,dm, \ \ \varphi\in\mathcal{C}(\Omega).
\]
\end{proof}

Since the representing measure of $U^{\#}$ is unique (see Proposition \ref{p:mUassocuniq}), the following is immediate from Proposition \ref{p:mUismUassoc}.

\begin{cor} \label{c:mofUisunique}
Let $X$ and $Y$ be Banach spaces and let $\Omega$ be a compact Hausdorff space. Let $1\leq p\leq\infty$. Then the representing measure $m:\Sigma\rightarrow \mathcal{L}(X,Y^{**})$ of an operator $U\in\mathcal{L}(\mathcal{C}_{p}(\Omega, X), Y)$ is unique.
\end{cor}

In view of Proposition \ref{p:mUismUassoc} and Corollary \ref{c:formula-m-mu}, the next corollary is immediate; the operators $J$ and $P$ were introduced before Theorem \ref{t:hatS}.

\begin{cor} \label{c:relation-m-m-mu}
Let $X$ and $Y$ be Banach spaces and let $\Omega$ be a compact Hausdorff space. Let $1\leq p\leq\infty$. Assume that $U\in\mathcal{L}(\mathcal{C}_{p}(\Omega,X),Y)$, let $m:\Sigma\rightarrow \mathcal{L}(X,Y^{**})$ be its representing measure, and let $\mu:\Sigma\rightarrow\mathcal{L}(X,Y)^{**}$ be the (classical) representing measure of the associated operator $U^{\#}\in\mathcal{L}(\mathcal{C}(\Omega),\mathcal{L}(X,Y))$. Then $m(E)=PJ^{**}\mu(E)$ for all $E\in\Sigma$.

Moreover, if $U^{\#}$ is weakly compact, then $m$ takes its values in $\mathcal{L}(X,Y)$, and the measures $m$ and $\mu$ coincide. In this case, the measure $m:\Sigma\rightarrow \mathcal{L}(X,Y)$ is countably additive and regular.
\end{cor}

\begin{rmk}\label{r:relation-m-mu}
Concerning the classical case of $\mathcal{C}(\Omega, X)=\mathcal{C}_{\infty}(\Omega,X)$, Corollary \ref{c:relation-m-m-mu} provides, for the first time in the literature, a general formula connecting the representing measure $m:\Sigma\rightarrow\mathcal{L}(X,Y^{**})$ of $U\in\mathcal{L}(\mathcal{C}(\Omega,X),Y)$ and the classical representing measure $\mu:\Sigma\rightarrow\mathcal{L}(X,Y)^{**}$ of $U^{\#}\in\mathcal{L}(\mathcal{C}(\Omega),\mathcal{L}(X,Y))$. In this sense, let us point out the partial result due to Dinculeanu (see \cite[Theorems 4 and 5]{Di2}, or, e.g., \cite[p. 388, Theorem 4]{D}): if $U\in\mathcal{L}(\mathcal{C}(\Omega,X),Y)$ and $U^{\#}\in\mathcal{L}(\mathcal{C}(\Omega),\mathcal{L}(X,Y))$ are \emph{dominated} operators, then $m$ takes its values in $\mathcal{L}(X,Y)$, and the measures $m$ and $\mu$ coincide.
\end{rmk}

Recall that a vector measure $m:\Sigma\rightarrow \mathcal{L}(X,Y^{**})$ is \emph{weakly regular} if $m_{y^{*}}:\Sigma\rightarrow X^{*}$ is regular for all $y^{*}\in Y^{*}$ (see, e.g., \cite[p. 181]{DU}). The following result extends the classical Dinculeanu--Singer theorem (see, e.g., \cite[p. 182]{DU}) in all its aspects (see Corollary \ref{c:Dinc-Sing-th} and the paragraph preceding it).

\begin{thm} \label{t:Dinc-Sing-general}
Let $X$ and $Y$ be Banach spaces and let $\Omega$ be a compact Hausdorff space. Let $1\leq p \leq \infty$.

$($\emph{a}$)$ Every operator $U\in\mathcal{L}(\mathcal{C}_{p}(\Omega,X),Y)$ has a unique representing measure $m:\Sigma\rightarrow \mathcal{L}(X,Y^{**})$. This measure coincides with the representing measure of its associated operator $U^{\#}\in\mathcal{L}(\mathcal{C}(\Omega),\mathcal{L}(X,Y))$.

$($\emph{b}$)$ Assume that $m:\Sigma\rightarrow \mathcal{L}(X,Y^{**})$ is a bounded vector measure. Then, there exists an operator $U\in\mathcal{L}(\mathcal{C}_{p}(\Omega,X),Y)$ such that $m$ is its representing measure if and only if for all $y^{*}\in Y^{*}$,
\[
I_{y^{*}}\in\mathcal{P}_{p'}(\mathcal{C}(\Omega), X^{*}),
\]
and the map $Y^{*}\rightarrow \mathcal{P}_{p'}(\mathcal{C}(\Omega), X^{*})=\mathcal{C}_{p}(\Omega, X)^{*}$, $y^{*}\mapsto I_{y^{*}}$, is linear, bounded, and weak*-to-weak* continuous.

In this case, $I_{y^{*}}= U^{*}y^{*}$ for all $y^{*}\in Y^{*}$, $U^{**}(\chi_{E}\otimes x) = m(E)x$ for all $E\in\Sigma$ and $x\in X$, $\|U\|=\|m\|_{p'}(\Omega)$, $\|U^{\#}\| = \|m\|(\Omega)$, and $m$ is a weakly regular measure if $p>1$.
\end{thm}

\begin{proof}
(a) A representing measure for an operator $S\in\mathcal{L}(\mathcal{C}(\Omega),\mathcal{L}(X,Y))$ (the associated operator $U^{\#}$ is of this type) always exists (see Remark \ref{r:constructionofm}). Theorem \ref{t:mUassocismU} and Proposition \ref{p:mUismUassoc} show that the representing measures of $U$ and $U^{\#}$ coincide. The measure is unique by Corollary \ref{c:mofUisunique}.

(b) Let $m:\Sigma\rightarrow \mathcal{L}(X,Y^{**})$ be the representing measure of $U\in\mathcal{L}(\mathcal{C}_{p}(\Omega,X),Y)$. Then, by (a) and Theorem \ref{t:mUassocismU}, $m$ is also the representing measure of $U^{\#}\in\mathcal{L}(\mathcal{C}(\Omega),\mathcal{L}(X,Y))$,  $I_{y^{*}}=U^{*}y^{*}\in\mathcal{P}_{p'}(\mathcal{C}(\Omega),X^{*})$ for all $y^{*}\in Y^{*}$, $U^{**}(\chi_{E}\otimes x) = m(E)x$ for all $E\in\Sigma$ and $x\in X$, and $\|U\|=\|m\|_{p'}(\Omega)$. In particular, $U^{*}: y^{*}\mapsto I_{y^{*}}$ is linear, bounded, and weak*-to-weak* continuous. This shows the ``only if'' part.

For the ``if'' part, let $m:\Sigma\rightarrow \mathcal{L}(X,Y^{**})$ be a bounded vector measure. Denote by $V$ the map given by the assumption, i.e.,
\[
V:Y^{*}\rightarrow \mathcal{P}_{p'}(\mathcal{C}(\Omega), X^{*})=\mathcal{C}_{p}(\Omega, X)^{*},\ \ y^{*}\mapsto I_{y^{*}}.
\]
Since $V$ is weak*-to-weak* continuous, there exists an operator $U\in\mathcal{L}(\mathcal{C}_{p}(\Omega,X),Y)$ such that $U^{*}=V$.

We only need to show that
\[
\int_{\Omega}\varphi\,dm = U^{\#}\varphi,\ \ \varphi\in\mathcal{C}(\Omega),
\]
because then $m:\Sigma\rightarrow\mathcal{L}(X,Y^{**})$ is the representing measure of $U^{\#}\in\mathcal{L}(\mathcal{C}(\Omega),\mathcal{L}(X,Y))$, hence also the representing measure of $U$ (see Theorem \ref{t:mUassocismU}).

For every $\varphi\in\mathcal{C}(\Omega)$, $x\in X$, and $y^{*}\in Y^{*}$, using \eqref{eq:intmystar}, we have that
\[
\langle y^{*}, \Big(\int_{\Omega}\varphi\,dm\Big)x \rangle = \langle \int_{\Omega}\varphi\,dm ,x\otimes y^{*} \rangle = \langle x, \int_{\Omega}\varphi\,dm_{y^{*}} \rangle
\]
\[
= \langle x, I_{y^{*}}\varphi \rangle = \langle x, (Vy^{*})\varphi \rangle = \langle \varphi x, Vy^{*} \rangle \]
\[
= \langle \varphi x, U^{*}y^{*} \rangle = \langle U(\varphi x), y^{*} \rangle = \langle (U^{\#}\varphi)x,y^{*} \rangle.
\]
This proves that $m$ is the representing measure of $U^{\#}$, as desired.

For the ``in this case'' part, the first three claims were already observed above. Since $m$ is also the representing measure of $U^{\#}\in\mathcal{L}(\mathcal{C}(\Omega),\mathcal{L}(X,Y))$, by Corollary \ref{c:normSandm}, we have that $\|U^{\#}\|=\|m\|(\Omega)$. Concerning the remaining claim about the weak regularity, recall that $m_{y^{*}}$ is the representing measure of $I_{y^{*}}\in\mathcal{P}_{p'}(\mathcal{C}(\Omega),X^{*})$ for every $y^{*}\in Y^{*}$. If $p>1$, then $p'<\infty$, and $I_{y^{*}}$ is a weakly compact operator (see, e.g., \cite[p. 50, Theorem 2.17]{DJT}). Therefore, $m_{y^{*}}$ is regular (see, e.g., \cite[p. 159, Corollary 14]{DU}) for all $y^{*}\in Y^{*}$.
\end{proof}

The next example shows that, for $p=1$, the measure $m$ in Theorem \ref{t:Dinc-Sing-general} is not weakly regular in general.

\begin{exam} \label{ex:regularityofm}
Let $X$ be a Banach space and let $\Omega$ be a compact Hausdorff space such that there exists a non-weakly compact operator $S\in\mathcal{L}(\mathcal{C}(\Omega),X^{*})$. Then, there exists an operator $U\in\mathcal{L}(\mathcal{C}_{1}(\Omega,X),\mathds{K})$ such that its representing measure $m:\Sigma\rightarrow\mathcal{L}(X,\mathds{K})=X^{*}$ is not weakly regular.

\begin{proof}
Let $S\in\mathcal{L}(\mathcal{C}(\Omega),X^{*})$ be a non-weakly compact operator. Then its representing measure $m:\Sigma\rightarrow X^{***}$ is not regular (see, e.g., \cite[p. 159, Corollary 14]{DU}).

Since, as is well known, $\mathcal{L}(\mathcal{C}(\Omega),X^{*})$ is canonically isometrically isomorphic to $(\mathcal{C}(\Omega)\otimes_{\pi}X)^{*}=(\mathcal{C}(\Omega)\hat{\otimes}_{\pi}X)^{*}$ and $\pi = d_{1}$, there exists a unique operator $U\in\mathcal{L}(\mathcal{C}(\Omega)\hat{\otimes}_{d_{1}}X,\mathds{K}) = \mathcal{L}(\mathcal{C}_{1}(\Omega,X),\mathds{K})$ such that $S=U^{\#}$. Then $m$ is the representing measure of $U$ because $m$ is the representing measure of $U^{\#}=S$ (see Theorem \ref{t:mUassocismU}). We know that $U^{*}\in\mathcal{L}(\mathds{K}^{*},\mathcal{C}_{1}(\Omega, X)^{*}) = \mathcal{L}(\mathds{K},\mathcal{L}(\mathcal{C}(\Omega), X^{*}))$ and, by Theorem \ref{t:mUassocismU}, $I_{1}=U^{*}1$. On the other hand, for every $\varphi\in\mathcal{C}(\Omega)$ and $x\in X$, we have
\[
((U^{*}1)\varphi)x = \langle \varphi x, U^{*}1\rangle = 1\,U(\varphi x) = (U^{\#}\varphi)x = (S\varphi)x,
\]
meaning that $U^{*}1=S$. Therefore, $I_{1} = S$, and its representing measure, which is $m$, is not regular. Hence, the representing measure of $U$ is not weakly regular.
\end{proof}
\end{exam}

Since $\mathcal{C}_{\infty}(\Omega,X) = \mathcal{C}(\Omega,X)$ and, hence, $\mathcal{P}_{1}(\mathcal{C}(\Omega), X^{*})=\mathcal{C}_{\infty}(\Omega,X)^{*}=\mathcal{C}(\Omega, X)^{*}$, Theorem \ref{t:Dinc-Sing-general} immediately yields the classical Dinculeanu--Singer theorem. Notice that, for every $y^{*}\in Y^{*}$, we can identify $I_{y^{*}}\in\mathcal{P}_{1}(\mathcal{C}(\Omega), X^{*})=\mathcal{C}(\Omega, X)^{*}$ with its (unique) representing measure $m_{y^{*}}:\Sigma\rightarrow X^{*}$. Let us stress that below we do not need to know about the Riesz--Singer representation of $\mathcal{C}(\Omega,X)^{*}$ as $r\,ca\,bv(\Sigma,X^{*})$. However, we get the regularity of the measures $m_{y^{*}}$ from our general setting. We also obtain the countable additivity of $m_{y^{*}}$ thanks to the Bartle--Dunford--Schwartz theorem (because $I_{y^{*}}$ are weakly compact). Moreover, the measures $m_{y^{*}}$ are of bounded variation (because they are the representing measures of absolutely summing operators $I_{y^{*}}$ (see, e.g., \cite[p. 162, Theorem 3]{DU})). So that, in the special case when $Y=\mathds{K}$, also the Riesz--Singer theorem is contained in Corollary \ref{c:Dinc-Sing-th} below (recall that, for a vector measure $m:\Sigma\rightarrow X^{*}$, one has $\|m\|_{1}(\Omega) = |m|(\Omega)$, the variation of $m$ on $\Omega$ (see, e.g., \cite[p. 54, Proposition 4]{D})).

\begin{cor} [cf.~{the Dinculeanu--Singer theorem, e.g., \cite[p. 182]{DU}}] \label{c:Dinc-Sing-th}
Let $X$ and $Y$ be Banach spaces and let $\Omega$ be a compact Hausdorff space.

$($\emph{a}$)$ Every operator $U\in\mathcal{L}(\mathcal{C}(\Omega, X),Y)$ has a unique representing measure $m:\Sigma\rightarrow\mathcal{L}(X,Y^{**})$. This measure coincides with the representing measure of its associated operator $U^{\#}\in\mathcal{L}(\mathcal{C}(\Omega),\mathcal{L}(X,Y))$.

$($\emph{b}$)$ Assume that $m:\Sigma\rightarrow \mathcal{L}(X,Y^{**})$ is a bounded vector measure. Then, there exists an operator $U\in\mathcal{L}(\mathcal{C}(\Omega,X),Y)$ such that $m$ is its representing measure if and only if for all $y^{*}\in Y^{*}$,
\[
m_{y^{*}}\in\mathcal{C}(\Omega, X)^{*},
\]
and the map $Y^{*}\rightarrow \mathcal{C}(\Omega, X)^{*}$, $y^{*}\mapsto m_{y^{*}}$, is linear, bounded, and weak*-to-weak* continuous.

In this case, $m_{y^{*}}:\Sigma\rightarrow X^{*}$ is countably additive and of bounded variation, $m_{y^{*}}= U^{*}y^{*}$ for all $y^{*}\in Y^{*}$, $\|U\|=\|m\|_{1}(\Omega)$, and $m$ is weakly regular.
\end{cor}

\begin{rmk} \label{r:proofDinculeanu--Singerth}
As we mentioned in the Introduction, in our general treatise, we did not follow any of the traditional proofs of the Dinculeanu--Singer theorem. The traditional proofs are of two types, although both extend methods of the classical proof of the Bartle--Dunford--Schwartz theorem in \cite[Theorem 3.1]{BDS} or \cite[p. 492, Theorem 2]{DS}. The proofs, e.g., by Batt and K\"{o}nig \cite{BK}, Dinculeanu \cite{Di}, \cite[pp. 398--399, Theorem 9]{D}, Foia\c{s} and Singer \cite{FS}, Swong \cite{Swong}, Tucker \cite{Tucker}, essentially rely on the Riesz--Singer representation theorem. The proofs, e.g., by Brooks and Lewis \cite{BrLe}, and Diestel and Uhl \cite[pp. 181--182]{DU} use ``the device of embedding isometrically the simple functions in $\mathcal{C}(\Omega,X)^{**}$ and thus reducing the problem to utilizing the representing theorem for operators $L\in\mathcal{L}(\mathcal{B}(\Sigma, X),Y)$, which can be easily established''. We quoted Brooks and Lewis \cite[p. 139]{BrLe} here; the mentioned representing theorem can be found in Dinculeanu's book \cite[p. 145, Theorem 1]{D}.
\end{rmk}

\begin{rmk} \label{r:about hatU 2}
Batt and Berg \cite{BB} introduced the notion of the \emph{weak extension} of an operator $U\in\mathcal{L}(\mathcal{C}(\Omega,X),Y)$, which is precisely the integration operator $\hat{U}\in\mathcal{L}(\mathcal{B}(\Sigma,X),Y^{**})$ with respect to the representing measure $m:\Sigma\rightarrow \mathcal{L}(X,Y^{**})$ of $U$. They proved that $\|\hat{U}\|=\|U\|$, $\hat{U}(\chi_{E}x) = m(E)x$ for all $E\in\Sigma$ and $x\in X$ (see \cite[Theorem 1]{BB}), and that $\textrm{ran}\,m\subset \mathcal{L}(X,Y)$ if and only if $\textrm{ran}\,\hat{U}\subset Y$ (see \cite[Theorem 2]{BB}). However, as our Theorems \ref{t:def-int-dp} and \ref{t:def-int-U} clearly show, these are general properties of any integration operator $\hat{U}\in\mathcal{L}(\mathcal{B}(\Sigma,X),Y^{**})$ and its restriction $U:=\hat{U}|_{\mathcal{C}(\Omega,X)}$. Moreover, even $\|\hat{U}\|=\|U\|=\|m\|_{1}(\Omega)$ and (by \eqref{eq:int-dp} and the ``moreover'' part of Theorem \ref{t:def-int-U}) $\hat{U}(\chi_{E}x) = m(E)x = U^{**}(\chi_{E}\otimes x)$ for all $E\in\Sigma$ and $x\in X$ in this general case.
\end{rmk}

\section{\texorpdfstring{Complements to the Dinculeanu--Singer theorem}{Complements to the Dinculeanu--Singer theorem}}\label{s5}

Let $X$ and $Y$ be Banach spaces and let $\Omega$ be a compact Hausdorff space. Let $1\leq p\leq\infty$. Let $S\in\mathcal{L}(\mathcal{C}(\Omega),\mathcal{L}(X,Y))$. In \cite{MOP2}, we studied the problem when does there exist an operator $U\in\mathcal{L}(\mathcal{C}_{p}(\Omega,X),Y)$ such that $S=U^{\#}$? In this section, we shall apply some result from \cite{MOP2} to prove some qualitative complements to Theorem \ref{t:Dinc-Sing-general}, the extension of the Dinculeanu--Singer theorem.

The idea behind the results below is as follows: the existence of an operator $U\in\mathcal{L}(\mathcal{C}_{p}(\Omega,X),Y)$ such that a given vector measure $m:\Sigma\rightarrow\mathcal{L}(X,Y^{**})$ is its representing measure is equivalent to the existence of an operator $S\in\mathcal{L}(\mathcal{C}(\Omega),\mathcal{L}(X,Y))$ such that $m$ is the representing measure of $S$ and such that $S=U^{\#}$. Notice that we shall not need Theorem \ref{t:Dinc-Sing-general} at all. Besides \cite{MOP2}, we shall rely on Theorem \ref{t:BDSt-S}, our extension of the Bartle--Dunford--Schwartz theorem, together with Theorem \ref{t:mUassocismU} and Proposition \ref{p:mUismUassoc}.

The next theorem also contributes to the classical Dinculeanu--Singer case when $p=\infty$.

\begin{thm} \label{t:Din-Sing and MOP2-1}
Let $X$ and $Y$ be Banach spaces and let $\Omega$ be a compact Hausdorff space. Let $1\leq p \leq \infty$. Assume that $m:\Sigma\rightarrow \mathcal{L}(X,Y^{**})$ is a bounded vector measure. Then, there exists an operator $U\in\mathcal{L}(\mathcal{C}_{p}(\Omega,X),Y)$ such that $m$ is its representing measure if and only if

$($\emph{i}$)$ for all $x\in X$,
\[
\langle y^{*},m_{x}(\cdot)\rangle \in \mathcal{C}(\Omega)^{*},\ \ y^{*}\in Y^{*},
\]
and the map $Y^{*}\rightarrow \mathcal{C}(\Omega)^{*}$, $y^{*}\mapsto \langle y^{*},m_{x}(\cdot)\rangle$, is linear, bounded and weak*-to-weak* continuous, and

$($\emph{ii}$)$ one of the following equivalent conditions holds:

$($\emph{a}$)$ there exists a constant $c>0$ such that, for all finite systems $(x_{i})_{i=1}^{n}\subset X$ and $(\varphi_{i})_{i=1}^{n}\subset \mathcal{C}(\Omega)$,
\[
\Big\|\Big(\int_{\Omega}\varphi_{i}\,dm_{x_{i}}\Big)\Big\|_{p'}^{w} \leq c\,\|(x_{i})\|_{\infty}\|(\varphi_{i})\|_{p'}^{w};
\]

$($\emph{b}$)$ there exists a constant $c>0$ such that, for all $(x_{i})\in \ell_{\infty}(X)$ and $(\varphi_{i})\in \ell_{p'}^{w}(\mathcal{C}(\Omega))$, and for all $n\in\mathds{N}$,
\[
\Big\|\Big(\int_{\Omega}\varphi_{i}\,dm_{x_{i}}\Big)_{i=n}^{\infty}\Big\|_{p'}^{w} \leq c\,\|(x_{i})_{i=n}^{\infty}\|_{\infty}\|(\varphi_{i})_{i=n}^{\infty}\|_{p'}^{w};
\]

$($\emph{c}$)$ if $(x_{i})\in \ell_{\infty}(X)$ and $(\varphi_{i})\in \ell_{p'}^{w}(\mathcal{C}(\Omega))$, then $(\int_{\Omega}\varphi_{i}\,dm_{x_{i}})\in\ell_{p'}^{w}(Y)$;

$($\emph{d}$)$ if $(x_{i})\in c_{0}(X)$ and $(\varphi_{i})\in \ell_{p'}^{w}(\mathcal{C}(\Omega))$ $($or $(x_{i})\in \ell_{\infty}(X)$ and $(\varphi_{i})\in \ell_{p'}^{u}(\mathcal{C}(\Omega))$$)$, then $(\int_{\Omega}\varphi_{i}\,dm_{x_{i}})\in\ell_{p'}^{u}(Y)$.
\end{thm}

\begin{proof}
We are going to use the following fact. Assume that $m$ is the representing measure of an operator $S\in\mathcal{L}(\mathcal{C}(\Omega),\mathcal{L}(X,Y))$. Since
\[
(S\varphi)x = (\int_{\Omega}\varphi\,dm)x = \int_{\Omega}\varphi\,dm_{x} \ \ \mbox{ for all } \varphi\in\mathcal{C}(\Omega) \mbox{ and } x\in X,
\]
by \cite[Corollary 3.4]{MOP2}, every condition included in (ii) is equivalent to the existence of an operator $U\in\mathcal{L}(\mathcal{C}_{p}(\Omega,X),Y)$ such that $U^{\#}=S$.

For the ``only if'' part, let $U\in\mathcal{L}(\mathcal{C}_{p}(\Omega,X),Y)$ be such that $m$ is its representing measure. By Proposition \ref{p:mUismUassoc}, $m$ is also the representing measure of its associated operator $U^{\#}\in\mathcal{L}(\mathcal{C}(\Omega),\mathcal{L}(X,Y))$, and, by the above fact, (ii) holds; condition (i) is immediate from Theorem \ref{t:BDSt-S}.

For the ``if'' part, condition (i) implies that there exists an operator $S\in\mathcal{L}(\mathcal{C}(\Omega),\mathcal{L}(X,Y))$ such that $m$ is its representing measure (see Theorem \ref{t:BDSt-S}). And, by the above fact, condition (ii) implies that there exists an operator $U\in\mathcal{L}(\mathcal{C}_{p}(\Omega,X),Y)$ such that $U^{\#} = S$. Then, by Theorem \ref{t:mUassocismU}, $m$ is also the representing measure of $U$.
\end{proof}

In the next theorem, we use \cite[Corollary 2.5]{MOP2}, which asserts that, for every operator $S\in\mathcal{L}(\mathcal{C}(\Omega), \mathcal{L}(X,Y))$, there exists an operator $U\in\mathcal{L}(\mathcal{C}_{1}(\Omega,X),Y)$ such that $U^{\#}=S$.

\begin{thm} \label{t:Din-Sing and MOP2-2}
Let $X$ and $Y$ be Banach spaces and let $\Omega$ be a compact Hausdorff space. Assume that $m:\Sigma\rightarrow \mathcal{L}(X,Y^{**})$ is a bounded vector measure. Then, there exists an operator $U\in\mathcal{L}(\mathcal{C}_{1}(\Omega,X),Y)$ such that $m$ is its representing measure if and only if there exists an operator $S\in\mathcal{L}(\mathcal{C}(\Omega), \mathcal{L}(X,Y))$ such that $m$ is its representing measure.
\end{thm}

\begin{proof}
The necessary condition is clear by taking $S=U^{\#}$ and applying Proposition \ref{p:mUismUassoc}. By \cite[Corollary 2.5]{MOP2}, for a given operator $S\in\mathcal{L}(\mathcal{C}(\Omega), \mathcal{L}(X,Y))$, there exists an operator $U\in\mathcal{L}(\mathcal{C}_{1}(\Omega,X),Y)$ such that $S=U^{\#}$. From Theorem \ref{t:mUassocismU}, the sufficient condition is clear.
\end{proof}

The following results, which are similar to Theorem \ref{t:Din-Sing and MOP2-2}, can be obtained using \cite[Corollaries 2.6 and 2.7]{MOP2} (instead of \cite[Corollary 2.5]{MOP2}).

\begin{thm} \label{t:Din-Sing and MOP2-3}
Let $X$ and $Y$ be Banach spaces such that $X^{*}$ is of cotype 2. Let $\Omega$ be a compact Hausdorff space. Assume that $m:\Sigma\rightarrow \mathcal{L}(X,Y^{**})$ is a bounded vector measure. Then, for every $p\leq 2$, there exists an operator $U\in\mathcal{L}(\mathcal{C}_{p}(\Omega,X),Y)$ such that $m$ is its representing measure if and only if there exists an operator $S\in\mathcal{L}(\mathcal{C}(\Omega), \mathcal{L}(X,Y))$ such that $m$ is its representing measure.
\end{thm}

\begin{thm} \label{t:Din-Sing and MOP2-4}
Let $X$ and $Y$ be Banach spaces such that $X^{*}$ is of cotype $q$, where $2\leq q <\infty$. Let $\Omega$ be a compact Hausdorff space. Assume that $m:\Sigma\rightarrow \mathcal{L}(X,Y^{**})$ is a bounded vector measure. Then, for every $p\leq q'$, there exists an operator $U\in\mathcal{L}(\mathcal{C}_{p}(\Omega,X),Y)$ such that $m$ is its representing measure if and only if there exists an operator $S\in\mathcal{L}(\mathcal{C}(\Omega), \mathcal{L}(X,Y))$ such that $m$ is its representing measure.
\end{thm}

\section*{Acknowledgements}

The research of Eve Oja was partially supported by institutional research funding IUT20-57 of the Estonian Ministry of Education and Research. The research of C\'{a}ndido Pi\~{n}eiro and Fernando Mu\~{n}oz was partially supported by the Junta de Andaluc\'{\i}a P.A.I. FQM-276.



\end{document}